\documentclass[11pt]{article}

\usepackage{setspace}

\usepackage{amsmath}
\usepackage{amssymb}
\usepackage{amsthm}
\usepackage{mathtools}
\usepackage{empheq}
\usepackage{enumitem}
\usepackage{titlesec}
\usepackage{graphicx}
\usepackage{caption}
\usepackage{subcaption}

\usepackage{diagbox}

\usepackage[T1]{fontenc}
\usepackage[utf8]{inputenc}

\usepackage{comment}

\usepackage{indentfirst}
\usepackage[top=1.25in,left=1.25in,right=1.25in]{geometry}
\newlength{\numone}
\newlength{\widone}
\newlength{\numtwo}
\newlength{\widtwo}
\newtheorem{thm}{Theorem}[section]
\newtheorem{lemma}[thm]{Lemma}
\newtheorem{prop}[thm]{Proposition}
\newtheorem{cor}[thm]{Corollary}
\newtheorem{defin}[thm]{Definition}

\newtheorem{alg}[thm]{Algorithm}

\numberwithin{equation}{section}

\author{\Large{Riccardo W. Maffucci}}
\newcommand{\Addresses}{{
		\footnotesize
		
		R.W.~Maffucci, \textsc{University of Coventry, United Kingdom CV1}\par\nopagebreak\vspace{-0.35cm}
		\textit{E-mail address}, R.W.~Maffucci: \texttt{riccardowm@hotmail.com}}}

\title{\Large{\uppercase{\bf Classification and construction of planar, 3-connected Kronecker products}}}

\date{}

\def\ho{H_{\text{odd}}}

\def\calC{\mathcal{C}}

\def\calP{\mathcal{P}}

\newcommand{\E}{\mathbb{E}}

\begin{document}
\titleformat{\section}
  {\Large\scshape}{\thesection}{1em}{}
\titleformat{\subsection}
  {\large\scshape}{\thesubsection}{1em}{}
\maketitle
\Addresses


\begin{abstract}
We give a complete classification of the Kronecker (i.e. direct) product graphs that are planar and $3$-connected (i.e. $3$-polytopal). They are all of the form 
\[H\wedge K_2,\]
where $H$ is a $2$-connected graph, possibly non-planar, and satisfying specific properties that we will describe.

Our proof is constructive, in the sense that we prescribe how to obtain all such graphs $H$, by adding a few edges in a specific way to a given planar, bipartite graph, that is either $3$-connected, or semi-hyper-$2$-connected.

Moreover, for $H$ planar, we also give a more precise characterisation of this graph, regarding the number of its odd regions, and how they intersect.

If $H\wedge K_2$ is a $3$-polytope, then we have $\delta(H\wedge K_2)=3$, so that the connectivity of $H\wedge K_2$ is $3$, and the connectivity of $H$ is either $2$ or $3$.

We also briefly discuss which Cartesian and strong products are $3$-polytopal.
\end{abstract}

\section{Introduction}
\subsection{Main Results}
A graph is planar if we can sketch it in the plane so that no edges cross, except possibly at vertices. For $k\geq 1$, a graph $G$ is $k$-connected if $|V(G)|>k$ and, however we remove fewer than $k$ vertices from $G$, the resulting graph is connected. We say that a graph has vertex connectivity $k$ if it is $k$-connected but not $k+1$-connected.

The planar, $3$-connected graphs are the ($1$-skeletons of) $3$-polytopes, sometimes also called polyhedra. The regions of a planar, $3$-connected graph are also referred to as faces, corresponding naturally to polyhedral faces. Two polyhedra are homeomorphic if and only if their respective $3$-polytopal graphs are isomorphic. More generally, the regions of a planar, $2$-connected graph are cycles (polygons). We call a region odd/even if it is bounded by a polygon with an odd/even number of sides.

A $3$-polytopal graph has a \textit{unique} embedding into a sphere, as observed by Whitney. It thus has a unique planar embedding, once the `external region' has been chosen. On the other hand, a planar graph $G$ of connectivity $2$ -- apart from a few special cases, e.g. simply a cycle -- may be immersed in the sphere/plane in more than one way. The polygons bounding the regions of $G$ may be different in distinct embeddings. However, $G$ has no odd regions if and only if it is bipartite, and this condition is independent of the embedding.

The Kronecker (also called `direct', or `tensor') product $G=H\wedge J$ of the graphs $H,J$ is defined as the graph of vertex set
\[V(G)=V(H)\times V(J)\]
and edge set
\[E(G)=\{(a,x)(b,y) : ab\in E(H) \text{ and } xy\in E(J)\}.\]

To simplify notation, we will sometimes denote a vertex $(a,x)$ in the Kronecker product as $ax$ when there is no possibility of confusion. The letters $G,H,J$ will indicate a graph, and $\calP$ a $3$-polytopal graph. We write $K_n$, $n\geq 1$ for the complete graph, and $P_n$, $n\geq 1$ for the simple path on $n$ vertices (and $n-1$ edges). We call a path \textit{odd/even} if $n$ is odd/even.

Prior literature \cite{farwal} investigated for which graphs $H,J$ the Kronecker product $H\wedge J$ is planar. The connectivity of Kronecker products has also received recent attention \cite{gujvum,wawu11,wanyan,ekikir}.

In this paper, we answer the following natural question. For which $H,J$ is $H\wedge J$ a $3$-polytope?

For example, let $J=K_2$. If $H$ is the $2n+1$-gonal prism, $n\geq 1$, then $H\wedge K_2$ is the $4n+2$-gonal prism. If $H$ is the $2n+1$-gonal pyramid, $n\geq 1$, then $H\wedge K_2$ is the dual of the $2n+1$-gonal antiprism, sometimes called the pseudo-double-wheel on $4n+4$ vertices (these are the graphs of kite-faced fair dice; for $n=1$ the pseudo-double-wheel is simply the cube). On the other hand, if $H$ is the $2n+2$-gonal pyramid, $n\geq 1$, then $H\wedge K_2$ is non-planar (see Corollary \ref{cor:4pyr} to follow). How about other $3$-polytopes $H$? Are there graphs $H$ that are not $3$-polytopes such that their Kronecker product with $K_2$ is a $3$-polytope? What about Kronecker products with something other than $K_2$? These are among the questions that we have completely answered in this paper.

It is known that, if $G=H\wedge J$ is planar, and $|V(J)|\leq |V(H)|$, then $J$ is $K_2$, $P_3$, or $P_4$ \cite[Theorem 5.3 and Proposition 5.4]{farwal}. We also have some information about how the connectivity of $H$ is related to that of $H\wedge K_2$ \cite{gujvum,wawu11,wanyan}.

For Kronecker products, when we impose planarity and $3$-connectivity together, we can say rather more. We start with the following.

\begin{prop}
	\label{prop:2}
	If
	\[G=H\wedge P_3,\]
	where $P_3$ is the path on three vertices, then $G$ is not $3$-polytopal.
\end{prop}
Proposition \ref{prop:2} will be proven in section \ref{sec:prop}. Clearly $H\wedge P_4$ contains a copy of $H\wedge P_3$, hence if the latter is non-planar then so is the former. Therefore, it remains to consider the case $J=K_2$.

Perhaps surprisingly, Kronecker products with $K_2$ are among the least understood in the literature \cite{samp75,bdgj09}. This may be related to the fact \cite{botmet} that $H\wedge K_2$ does not necessarily contain a subgraph isomorphic to $H$ (as opposed to $H\wedge K_3$, or e.g. any Cartesian product, that contains copies of each of its factors).

It is immediate to check that, if $H\wedge K_2$ is $3$-connected, then $H$ is $2$-connected, and $\delta(H)\geq 3$.

Another necessary condition that is easy to see, any Kronecker product with $K_2$ is a bipartite graph, hence our sought $\calP=H\wedge K_2$ are $3$-polytopes where all faces are even polygons. In fact, after a few more considerations, in section \ref{sec:pre} we will show the following.
\begin{prop}
	\label{prop:conn}
	If $H$ is a graph such that $\calP=H\wedge K_2$ is a $3$-polytope, then \[\delta(\calP)=\delta(H)=3.\]
	In particular, the connectivity of $\calP$ is $3$, and the connectivity of $H$ is either $2$ or $3$.
\end{prop}

Our first main result concerns the case of $H$ planar, with vertex connectivity $2$.

\begin{thm}
\label{thm:2}
Let $H$ be a planar graph with vertex connectivity $2$. Then $\calP=H\wedge K_2$ is a $3$-polytope if and only if all odd regions of $H$ except exactly two contain a $2$-cut, and moreover if $\{a,b\}$
is a $2$-cut in $H$, then $H-a-b$
has exactly two connected components, each of them containing at least one odd region.
\end{thm}

Theorem \ref{thm:2} will be proven in section \ref{sec:thm2}. Note that these conditions on $H$ clearly imply $\delta(H)\geq 3$. An illustration is given in Figure \ref{fig:thm2}. As remarked above, the number of odd regions of $H$ may depend on the embedding, however, we note that the conditions of Theorem \ref{thm:2} are independent of the embedding.

\begin{figure}[h!]
	\centering
	\begin{subfigure}{0.49\textwidth}
		\centering
		\includegraphics[width=5.5cm]{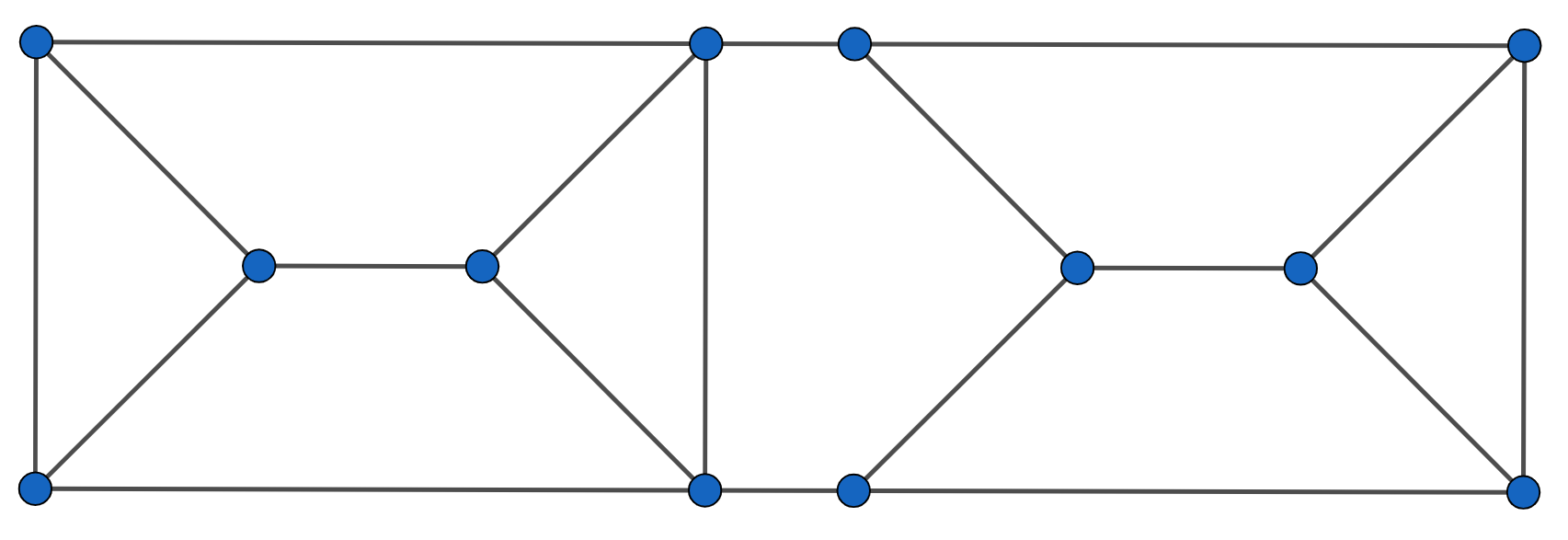}
		\caption{Example of $H$ as in Theorem \ref{thm:2}.}
		\label{fig:pl2c}
	\end{subfigure}
	\begin{subfigure}{0.49\textwidth}
		\centering
		\includegraphics[width=4cm]{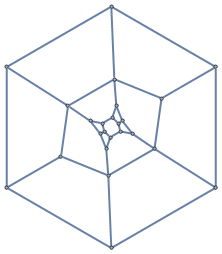}
		\caption{The $3$-polytope $H\wedge K_2$.}
		\label{fig:p0}
	\end{subfigure}
	\caption{Theorem \ref{thm:2}.}
	\label{fig:thm2}
\end{figure}

For the case where both $H$ and its Kronecker product with $K_2$ are $3$-polytopes, we can give a rather accurate description of $H$ in terms of the number of its odd faces and how they intersect.

\begin{thm}
	\label{thm:1}
	Let $H$ be a $3$-polytope and $O$ the set of its odd faces. Then $\calP=H\wedge K_2$ is a $3$-polytope if and only if one of the following holds:
	\begin{enumerate}[label=(\theenumi)]
		\item
		\label{eq:c1}
		$|O|=2$ and the two odd faces of $H$ are disjoint;
		\item
		\label{eq:c2}
		$|O|=4$, and any two elements of $O$ intersect in a vertex or edge, while any three elements of $O$ have empty intersection;
		\item
		\label{eq:c3}
		$|O|\geq 4$, all odd faces except one share a common vertex, and the remaining odd face has non-empty intersection with all other odd faces.
	\end{enumerate}
\end{thm}

Theorem \ref{thm:1} will be proven in sections \ref{sec:thm1n} and \ref{sec:thm1s}. Note that the three conditions in Theorem \ref{thm:1} are mutually exclusive. See Figure \ref{fig:pl3c} for examples of $3$-polytopes satisfying the three conditions.

\begin{figure}[h!]
	\centering
	\begin{subfigure}{0.49\textwidth}
		\centering
		\includegraphics[width=3cm]{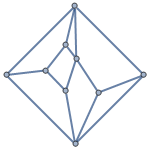}
		\includegraphics[width=3cm]{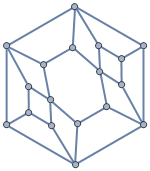}
		\caption{A $3$-polytope satisfying Condition \ref{eq:c1}, and its Kronecker product with $K_2$.}
		\label{fig:c1}
	\end{subfigure}
	\hfill
	\begin{subfigure}{0.49\textwidth}
		\centering
		\includegraphics[width=3cm]{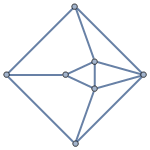}
		\includegraphics[width=3cm]{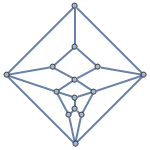}
		\caption{A $3$-polytope satisfying Condition \ref{eq:c3}, and its Kronecker product with $K_2$.}
		\label{fig:c3}
	\end{subfigure}
	\hfill
	\begin{subfigure}{0.99\textwidth}
		\centering
		\includegraphics[width=5.5cm]{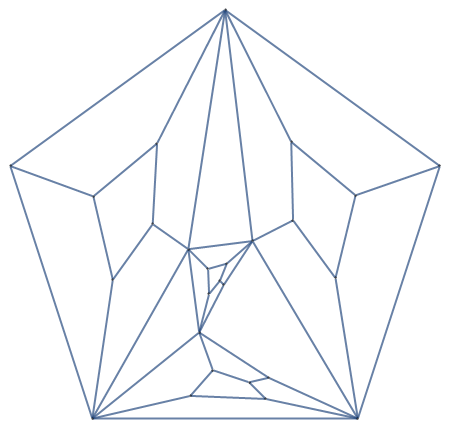}
		\hspace{0.5cm}
		\includegraphics[width=5.5cm]{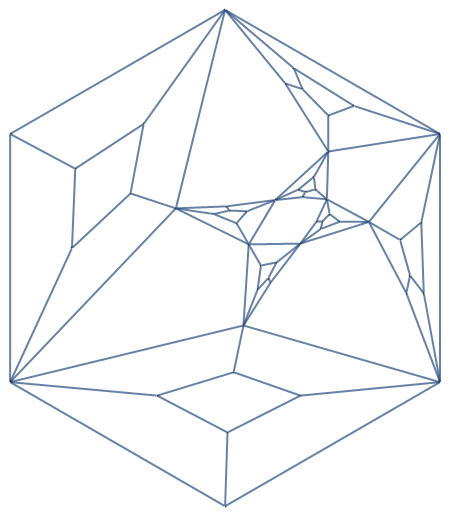}
		\caption{A $3$-polytope satisfying Condition \ref{eq:c2}, and its Kronecker product with $K_2$.}
		\label{fig:c2}
	\end{subfigure}
	\caption{Examples of $3$-polytopes $H$ as in Theorem \ref{thm:1}, and the corresponding $\calP=H\wedge K_2$.}
	\label{fig:pl3c}
\end{figure}

Together, Theorems \ref{thm:2} and \ref{thm:1} completely classify the planar graphs $H$ such that $H\wedge K_2$ is a $3$-polytope.

Perhaps surprisingly, it is possible for $H$ to be non-planar and $H\wedge K_2$ to be planar \cite{botmet,handpr}. Further, it is possible for $H$ to be non-planar and $H\wedge K_2$ to be $3$-polytopal. For instance, any $4n$-gonal prism, $n\geq 2$, may be written as $H_{n}\wedge K_2$, where $H_n$ is the non-planar graph sketched in Figure \ref{fig:hn}.

\begin{figure}[h!]
	\centering
	\includegraphics[width=5.5cm]{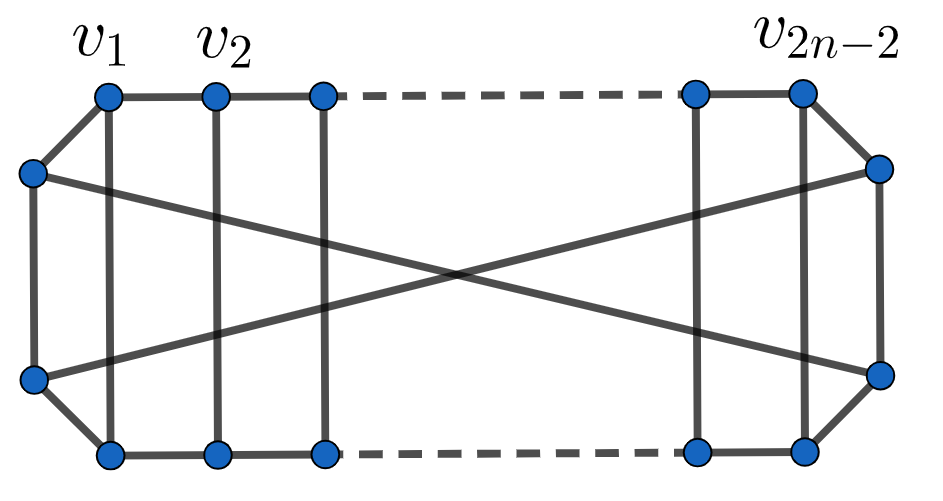}
	\caption{The non-planar graph $H_n$, $n\geq 2$.}
	\label{fig:hn}
\end{figure}

To treat the non-planar case, we introduce the following.
\begin{defin}
	We say that a graph $H$ is semi-hyper-$k$-connected if its vertex connectivity is $k$, and moreover every $k$-cut results in exactly $2$ connected components.
\end{defin}

We record that a planar graph $H$ is semi-hyper-$2$-connected if and only if, in any planar embedding of $H$, there exists a region containing every $2$-cut. For instance, the graph in Figure \ref{fig:pl2c} is planar and semi-hyper-$2$-connected. We have the following characterisation for non-planar graphs with $3$-polytopal Kronecker product with $K_2$.

\begin{thm}
	\label{thm:3}
	Suppose that $H$ is a $2$-connected, non-planar graph, with $\delta(H)\geq 3$. Then $\calP=H\wedge K_2$ is a $3$-polytope if and only if the following all hold. There exists a planar, bipartite spanning subgraph $H'$ of $H$, either $3$-connected of semi-hyper-$2$-connected, such that 
	\[H=H'+a_1b_1+a_2b_2+\dots+a_mb_m, \qquad a_1,b_1,a_m,b_m \text{ distinct,}\]
	where the graphs $H'+a_ib_i$, $1\leq i\leq m$ are not bipartite; there exists a planar embedding of $H'$ such that the vertices
	\begin{equation}
	\label{eq:vertabm}
	a_1,a_2,\dots,a_m,b_1,b_2,\dots,b_m
	\end{equation}
	lie in this order on a region $r$; such $r$ contains every $2$-vertex-cut in $H'$, and none of these $2$-vertex-cuts lie on $r$ between $b_m$ and $a_1$, or $a_m$ and $b_1$ extrema included.
\end{thm}

Theorem \ref{thm:3} will be proven in section \ref{sec:thm3}. For instance in $H_n$ of Figure \ref{fig:hn}, a feasible $H'$ is obtained by deleting the two diagonal edges. Another example is given in Figure \ref{fig:npl}. To clarify the last condition in Theorem \ref{thm:3}, if \[v_1,v_2,\dots,v_{V},a_1,a_2,\dots,a_m,w_1,w_2,\dots,w_{W},b_1,b_2,\dots,b_m\]
lie in this order on the contour of $r$, then no $2$-cut set of $H'$ is a subset of
\[\text{either } \{b_m,v_1,v_2,\dots,v_{V},a_1\} \quad \text{ or }\{a_m,w_1,w_2,\dots,w_{W},b_1\}.\]

\begin{figure}[h!]
	\centering
	\begin{subfigure}{0.49\textwidth}
		\centering
		\includegraphics[width=4cm]{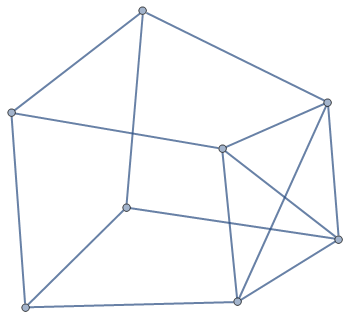}
		\caption{Example of $H$ as in Theorem \ref{thm:3}. As subgraph $H'$ we can simply take the cube.}
		\label{fig:np}
	\end{subfigure}
	\begin{subfigure}{0.49\textwidth}
		\centering
		\includegraphics[width=4cm]{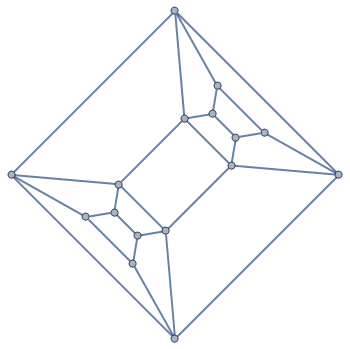}
		\caption{The $3$-polytope $H\wedge K_2$.}
		\label{fig:nptk2}
	\end{subfigure}
	\caption{Theorem \ref{thm:3}.}
	\label{fig:npl}
\end{figure}

Having covered all cases, we conclude that the graph $\calP=H\wedge J$ is a $3$-polytope if and only if (up to reordering) $J=K_2$ and $H$ is as in Theorems \ref{thm:2}, \ref{thm:1}, or \ref{thm:3}.

In fact, Theorem \ref{thm:3} is a special case of the following.
\begin{thm}
	\label{thm:4}
The graph $\calP=H\wedge J$, where $|V(J)|\leq|V(H)|$, is a $3$-polytope if and only if the following are all satisfied. We have $J=K_2$ and $\delta(H)\geq 3$; $H$ contains a planar, bipartite spanning subgraph $H'$, either $3$-connected of semi-hyper-$2$-connected, such that
\[H=H'+a_1b_1+a_2b_2+\dots+a_mb_m, \qquad a_1,b_1,a_m,b_m \text{ distinct,}\]
where the graphs $H'+a_ib_i$, $1\leq i\leq m$ are not bipartite; there exists a planar embedding of $H'$ such that the endpoints of $a_1b_1$, $a_2b_2$, \dots, $a_mb_m$ all lie on a region $r$, either in the order
\begin{equation}
	\label{eq:ord1}
a_1,a_2,\dots,a_m,b_m,b_{m-1},\dots,b_1,
\end{equation}
or in the order
\begin{equation}
	\label{eq:ord2}
	a_1,a_2,\dots,a_m,b_1,b_2,\dots,b_m;
\end{equation}
such $r$ contains every $2$-vertex-cut in $H'$, and none of these $2$-vertex-cuts lie on $r$ between $b_1$ and $a_1$, or $a_m$ and $b_m$ in \eqref{eq:ord1}, or between $b_m$ and $a_1$, or $a_m$ and $b_1$ in \eqref{eq:ord2}, extrema included in each case.
\end{thm}

Theorem \ref{thm:4} will be proven in section \ref{sec:thm3}.

The disadvantage of this compact formulation is of course, that in the planar case information on the number and adjacencies of the odd regions/faces of $H$ is not explicit (recall Theorems \ref{thm:2} and \ref{thm:1}).

On the other hand, the compact formulation of Theorem \ref{thm:4} has the further advantage of being constructive, in the following sense.

\begin{alg}
	\label{alg:1}
Every graph $H$ such that $H\wedge K_2$ is a $3$-polytope may be constructed in the following way. We start with any planar, bipartite $H'$, that is either $3$-connected of semi-hyper-$2$-connected. We then add edges to $H'$ respecting the conditions in Theorem \ref{thm:4}, while ensuring that $\delta(H)\geq 3$.
\end{alg}
In Algorithm \ref{alg:1}, in case $H'$ is semi-hyper-$2$-connected, the region $r$ of $H'$ containing the endpoints of the edges to add must be the one containing all $2$-cuts, whereas if $H'$ is $3$-connected, $r$ may be chosen arbitrarily.

\subsection{Plan of the paper}
In section \ref{sec:prop}, we will prove Proposition \ref{prop:2}, implying that if a Kronecker product is a $3$-polytope, then one of the factors is $K_2$. The other factor $H$ is thus $2$-connected. For $H$ planar of connectivity $2$, we will prove Theorem \ref{thm:2} in section \ref{sec:thm2}. For $H$ planar and $3$-connected, we will prove Theorem \ref{thm:1} in sections \ref{sec:thm1n} and \ref{sec:thm1s}. For $H$ non-planar, we will prove Theorem \ref{thm:4}, implying Theorem \ref{thm:3}, in section \ref{sec:thm3}. This will conclude the proofs of the main results.

To build $H\wedge K_2$, we will be using the subgraph of $H$ mentioned in Theorems \ref{thm:3} and \ref{thm:4}, or a similar construction. In general, one first deletes from $H$ vertices and/or edges, to construct a bipartite graph $H'$ (cf. the odd cycle packing problem \cite{gkpt09,krodd1}). Next, one sketches two copies of $H'$ in the plane, yielding $H'\wedge K_2$ (see Lemma \ref{lem:h-v} to follow). Finally, for each deleted edge $ab\in E(H)$, one adds to $H'\wedge K_2$ the edges $(a,x)(b,y)$ and $(a,y)(b,x)$, where $x,y$ are the vertices of $K_2$. This produces $H\wedge K_2$. Our aim is to choose the vertices and edges to delete in such a way to uncover necessary and/or sufficient conditions on $H$ for $H\wedge K_2$ to be a $3$-polytope.

To conclude this introduction, we will now briefly inspect the other main graph products, such as Cartesian (also called `box') $\square$ and strong $\boxtimes$ products  \cite{handpr}. Our question here is similar: which Cartesian and strong graph products are $3$-polytopal? Much more is known in the literature about the planarity of these products, hence this question is readily answered, as stated in the following section, and proven in Appendix \ref{sec:appa}.

\subsection{Other graph products}
\label{sec:cs}
Three of the most commonly studied graph products are the Kronecker $G_1=H\wedge J$, Cartesian $G_2=H\square J$, and strong $G_3=H\boxtimes J$. Their definitions are
\[V(G_1)=V(G_2)=V(G_3)=V(H)\times V(J),\]
and
\begin{align*}
E(G_1)&=\{(a,x)(b,y) : ab\in E(H) \text{ and } xy\in E(J)\},
\\E(G_2)&=\{(a,x)(b,y) : (a=b \text{ and } xy\in E(J)) \text{ or } (ab\in E(H) \text{ and } x=y)\},
\\E(G_3)&=E(G_1)\cup E(G_2).
\end{align*}
 Sometimes in the literature they are denoted by $\times$, $\square$, and $\boxtimes$, with the mnemonic of what the respective products of two copies of $K_2$ could look like when sketched in the plane. We have elected to denote the Kronecker product by $\wedge$, reserving the symbol $\times$ for the Cartesian product of sets. This also avoids another possible confusion, since some texts use $\times$ for the Cartesian product of graphs.
\begin{prop}
	\label{prop:cart}
	The $3$-polytopes that are the Cartesian product of two graphs are exactly the products of $K_2$ and an outerplanar, $2$-connected graph, and the products of a path and a polygon.
\end{prop}
The proof of Proposition \ref{prop:cart} follows readily from \cite{behmah}, and may be found in Appendix \ref{sec:appa}. Proposition \ref{prop:cart} identifies two families of $3$-polytopes that are Cartesian products of graphs, namely the products of $K_2$ and an outerplanar, $2$-connected graph (e.g. Figure \ref{fig:cart1}), and the products of a path and a polygon (e.g. Figure \ref{fig:cart2}). The intersection of these two families is the set of prisms. By the way, an outerplanar graph is $2$-connected if and only if it is Hamiltonian. For the relation between Cartesian products and $d$-polytopes, we refer the interested reader to \cite{pfpisa}.

\begin{figure}[h!]
	\centering
	\begin{subfigure}{0.29\textwidth}
		\centering
		\includegraphics[width=4cm]{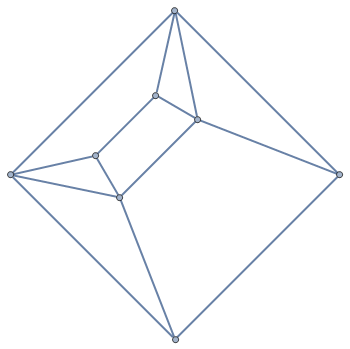}
		\caption{Cartesian product of the diamond graph and $K_2$.}
		\label{fig:cart1}
	\end{subfigure}
\hspace{0.75cm}
	\begin{subfigure}{0.27\textwidth}
		\centering
		\includegraphics[width=4cm]{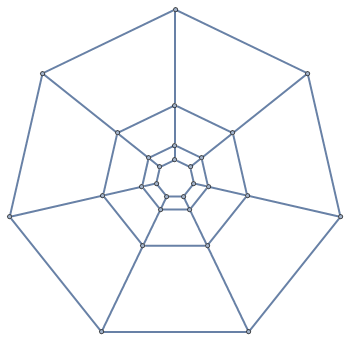}
		\caption{Cartesian product of $P_4$ and the heptagon.}
		\label{fig:cart2}
	\end{subfigure}
\hspace{0.75cm}
	\begin{subfigure}{0.27\textwidth}
		\centering
		\includegraphics[width=4cm]{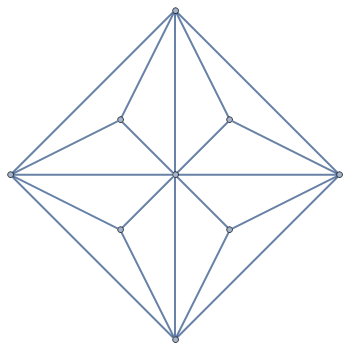}
		\caption{The $3$-polytope $P_3\boxtimes P_3$.}
		\label{fig:str}
	\end{subfigure}
	\caption{Propositions \ref{prop:cart} and \ref{prop:strong}.}
	\label{fig:cartstr}
\end{figure}

\begin{prop}
\label{prop:strong}
The graph $\calP=H\boxtimes J$ is a $3$-polytope if and only if $\calP=K_4=K_2\boxtimes K_2$ or $\calP=P_3\boxtimes P_3$, where $P_3$ is the simple path on three vertices.
\end{prop}
The proof of Proposition \ref{prop:strong} follows readily from \cite{jhaslu}, and may be found in Appendix \ref{sec:appa} to follow. The graph $P_3\boxtimes P_3$ is depicted in Figure \ref{fig:str}.


\section{Proof of Proposition \ref{prop:2}}
\label{sec:prop}
We start by showing that if a Kronecker product is $3$-polytopal, then one of the factors in $K_2$.
\begin{proof}[Proof of Proposition \ref{prop:2}]
	We will label
	\[E(P_3)=\{xy,yz\}\]
	throughout this proof. By contradiction, let $G=H\wedge P_3$ be a $3$-polytope. By \cite{botmet}, $H$ contains an odd cycle. Our first goal is to show that in fact $H$ contains at least two odd cycles.
	
	Suppose for the moment that $H$ has a separating vertex $a$, and call $H_1$ a component of $H-a$. We write
	\[\overline{H_1}:=\langle V(H_1)\cup \{a\}\rangle\]
	for the corresponding generated subgraph. In $G$ we have the $3$-cut
	\[\{ax,ay,az\}.\]
	Assume that $\overline{H_1}$ has no odd cycles. Since $P_3$ has vertices of degree one, and $G$ is a $3$-polytope, it follows that $\delta(H)\geq 3$. We can thus take
	\[b\in V(\overline{H_1})\]
	at distance $2$ from $a$. Since $\overline{H_1}$ has no odd cycles, in $G-ay$ the vertices $by$ and $ax$ are in different connected components, contradiction. Therefore, if $H$ has a separating vertex, then it contains at least two odd cycles.
	
	On the other hand, let $H$ be $2$-connected. Then it has an open ear decomposition. By the handshaking lemma, unless $H$ is simply an odd cycle, the number of odd cycles in $H$ is even, and in particular there are at least two. The tensor product of an odd cycle and $P_3$ is not $3$-connected, hence we have shown that, in any case, $H$ contains at least two odd cycles $\calC_1,\calC_2$.
	
	Let $\calC_1,\calC_2$ be disjoint. There is a path from a vertex of $\calC_1$ to a vertex of $\calC_2$ containing no other vertices from either cycle. We write
	\[\calC_1=(u_1,u_2=a,u_3,\dots,u_{2l+1}), \ l\geq 1, \qquad \calC_2=(v_1,v_2=b,v_3,\dots,v_{2m+1}), \ m\geq 1,\]
	and
	\[a,w_1,w_2,\dots,w_{n},b=w_{n+1}, \quad n\geq 0\]
	is a path containing no vertices of either cycle save $a,b$ (Figure \ref{fig:uvw}).
	\begin{figure}[h!]
		\centering
		\includegraphics[width=13cm]{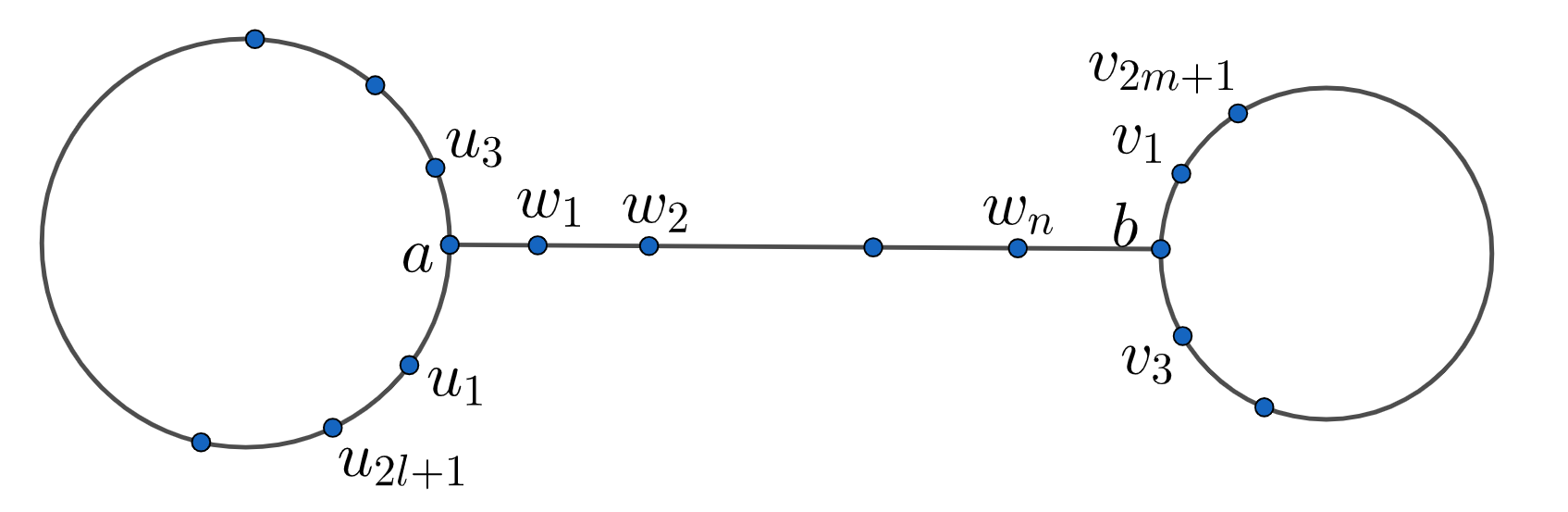}
		\caption{$H$ contains two disjoint odd cycles.}
	\label{fig:uvw}
	\end{figure}
	
	Then
	\[\{\{ax,ay,az\},\{u_1y,u_3y,w_1y\}\}\]
	determines a subgraph of $G$ homeomorphic to $K(3,3)$. Indeed, $ax$ and $az$ are each adjacent to all of $u_1y,u_3y,w_1y$ in $G$. As for $ay$, we have the internally disjoint paths
	\begin{eqnarray*}
		&ay,u_3x,u_4y,\dots,u_{2l+1}x,u_1y;
		\\&
		ay,u_1x,u_{2l+1}y,u_{2l}x,\dots,u_{3}y.
	\end{eqnarray*}
	To find a $(a,y)(w_1y)$-path in $G$ containing none of the above vertices, we consider the walk in $H$
	\[a,w_1,w_2,\dots,w_{n},b,v_3,v_4,\dots,v_{2m+1},v_1,b,w_n,w_{n-1},\dots,w_2,w_1,\]
	of odd length. We may thus take the corresponding
	\[ay,w_1x,\dots,w_2x,w_1y\]
	(see Figure \ref{fig:uvwtk2}).
	\begin{figure}[h!]
		\centering
		\includegraphics[width=14cm]{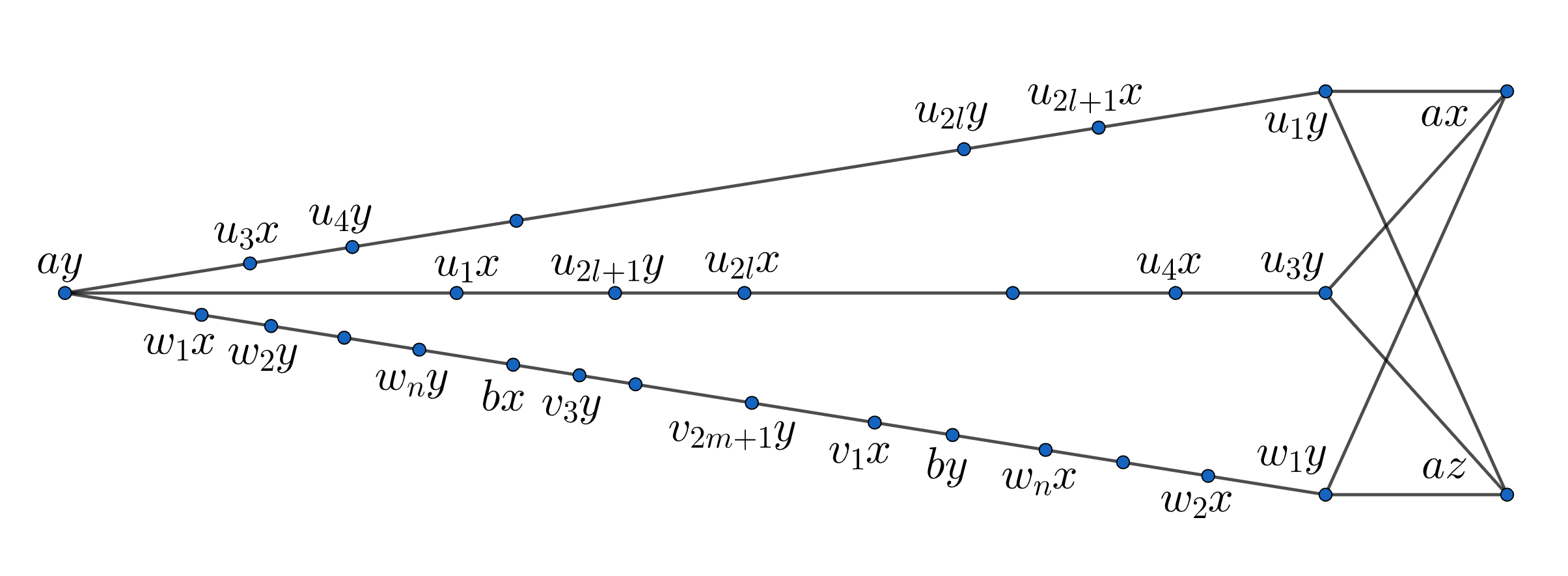}
		\caption{Subgraph of $H\wedge P_3$ homeomorphic to $K(3,3)$ -- case of two disjoint odd cycles, with $n$ even.}
		\label{fig:uvwtk2}
	\end{figure}
	
	Now instead let $\calC_1,\calC_2$ have exactly one common vertex,
	\[\calC_1=(u_1,u_2=a,u_3,\dots,u_{2l+1}), \ l\geq 1, \qquad \calC_2=(v_1,v_2=a,v_3,\dots,v_{2m+1}), \ m\geq 1.\]
	Here we take
	\[\{\{ax,ay,az\},\{u_1y,u_3y,v_1y\}\}.\]
	Again, $ax$ and $az$ are each adjacent to all of $u_1y,u_3y,v_1y$ in $G$, and we have internally disjoint paths
	\begin{eqnarray*}
		&ay,u_3x,u_4y,\dots,u_{2l+1}x,u_1y;
		\\&
		ay,u_1x,u_{2l+1}y,u_{2l}x,\dots,u_{3}y;
		\\&
		ay,v_3x,v_4y,\dots,v_{2m+1}x,v_1y.
	\end{eqnarray*}
	
	Finally, let $\calC_1,\calC_2$ have two or more common vertices. These cycles are hence in the same block of $H$. We consider the open ear decomposition for this planar block. As above we write
	\[\calC_1=(u_1,u_2,\dots,u_{2l+1}), \ l\geq 1, \qquad \calC_2=(v_1,v_2,\dots,v_{2m+1}), \ m\geq 1.\]
	
	We claim that up to relabelling one can find odd cycles $\calC_1,\calC_2$ in $H$ such that
	\begin{equation}
		\begin{cases}
			\label{eq:cases}
			u_i=v_j & \text{ for }2\leq i=j\leq I, \quad I\geq 3,
			\\
			u_i\neq v_j &\text{ otherwise.}
		\end{cases}
	\end{equation}
	Indeed, suppose that $u_2=v_2$, $u_i=v_j$, and \[\{u_3,u_4,\dots,u_{i-1}\}\cap \{v_3,v_4,\dots,v_{j-1}\}=\emptyset.\]
	Let
	\[\calC':=(u_2,u_3,\dots,u_{i-1},u_i,v_{j-1},v_{j-2},\dots,v_4,v_3)\]
	and
	\[\calC'':=(u_2,u_3,\dots,u_{i-1},u_i,v_{j+1},\dots,v_{2m+1},v_1).\]
	If $\calC'$ is an odd cycle, then the odd cycles $\calC_1,\calC'$ (after relabelling) satisfy \eqref{eq:cases}. If $\calC'$ is an even cycle, then the paths
	\[u_2,u_3,\dots,u_{i-1},u_i, \qquad \text{and}\qquad u_2,v_3,\dots,v_{j-1},u_i\]
	have the same parity, thus $\calC''$ is an odd cycle, hence $\calC_1,\calC''$ (after relabelling) satisfy \eqref{eq:cases}.
	
	Now that the claim is proven, we take $a=u_2$ and consider
	\[\{\{ax,az,u_It\},\{u_1y,u_3y,v_1y\}\},\]
	where
	\[\begin{cases}
		t=x \quad \text{ if } I \text{ is odd},
		\\
		t=y \quad \text{ if } I \text{ is even}.
	\end{cases}\]
	Here $ax$ and $az$ are each adjacent to all of $u_1y,u_3y,v_1y$ in $G$, and we have internally disjoint paths
	\begin{eqnarray*}
		&u_It,u_{I+1}t',\dots,u_{2l+1}x,u_1y;
		\\&
		u_It,v_{I+1}t',\dots,v_{2m+1}x,v_1y;
		\\&
		u_It,u_{I-1}t',\dots,u_{3}x,ay,u_1x,u_{2l+1}y,\dots,u_{I+1}t,u_It',u_{I-1}t,\dots,u_4x,u_3y.
	\end{eqnarray*}
	where $t'=y$ if $t=x$ and $t'=x$ if $t=y$.
\end{proof}

\section{Proof of Theorem \ref{thm:2}}
\label{sec:thm2}
Henceforth we consider only Kronecker products with $K_2$. The two vertices of $K_2$ will be always denoted by $x,y$. The symbol $\dot\cup$ denotes disjoint union of graphs.

\subsection{Preparatory results}
\label{sec:pre}
We begin with a few preliminaries. Some of these, or similar lemmas, have already appeared elsewhere, e.g. \cite{botmet,wanyan}. They will be useful in the proofs of Theorems \ref{thm:2}, \ref{thm:1}, \ref{thm:3}, and \ref{thm:4}.
\begin{lemma}
	\label{lem:h-v}
	For Kronecker products with $K_2$, the following hold.
	\begin{itemize}
		\item 
		If $H$ is a bipartite graph, then $H\wedge K_2=H\dot\cup H$.
		\item
		Let $H$ be a graph such that $G=H\wedge K_2$ is $3$-connected. Then for each $v\in V(H)$, the graph $H-v$ is not bipartite.
	\end{itemize}
\end{lemma}
\begin{proof}
	For the first statement, it suffices to note that each edge of $H\wedge K_2$ may be written as $(a,x)(b,y)$, where $a,b$ are vertices in distinct parts of the bipartite graph $H$. We now turn to the second statement. By contradiction, there exists a vertex $v$ of $H$ such that $H-v$ is bipartite. Then $(H-v)\wedge K_2$ has exactly $2$ components, thus $G-vx-vy$ is disconnected, contradiction.
\end{proof}

Recall that a subdivision of a graph $H$ is a sequence of operations where edges are replaced with simple paths of a certain length.
\begin{defin}
	We say that a subdivision of a graph is \emph{even} if each edge is replaced by a non-trivial path of even length.
\end{defin}

\begin{lemma}[cf. {\cite[Lemma 4]{bdgj09}}]
	\label{lem:subd}
	Let $H$ be a graph and $H'$ an even subdivision of $H$. Then $H\wedge K_2$ is planar if and only if $H'\wedge K_2$ is planar.
\end{lemma}
\begin{proof}
	It suffices to note that if $H'$ is an even subdivision of $H$, then $H'\wedge K_2$ is a subdivision of $H\wedge K_2$. The result now follows from Kuratowski's Theorem.
\end{proof}

\begin{cor}
	\label{cor:4pyr}
	Let $H$ be a graph containing an even subdivision of the square pyramid. Then $H\wedge K_2$ is non-planar.
\end{cor}
\begin{proof}
	One checks that the Kronecker product of the square pyramid and $K_2$ is non-planar, and then one applies Lemma \ref{lem:subd}.
\end{proof}

For example, any $2n+2$-gonal pyramid, $n\geq 1$, contains an even subdivision of the square pyramid, hence its Kronecker product with $K_2$ is non-planar.

We end this section by proving Proposition \ref{prop:conn}.
\begin{proof}[Proof of Proposition \ref{prop:conn}]
Let $H$ be a graph such that $\calP=H\wedge K_2$ is a $3$-polytopal $(p,q)$ graph. Since $\calP$ is bipartite, all of its faces are even. In particular, each is bounded by a cycle of length at least $4$. Double counting on the number of edges thus yields
\[2q\geq 4F,\]
where $F$ is the total number of faces in $\calP$. Via Euler's formula for planar graphs, we deduce
\[2q\leq 4p-8,\]
thus by the handshaking lemma $\calP$ has at least $8$ vertices of degree $3$. Since $\calP=H\wedge K_2$, then $H$ has at least $4$ vertices of degree $3$.
\end{proof}
Figures \ref{fig:c3} and \ref{fig:npl} depict extremal graphs with respect to this property, where $H$ has exactly $4$ vertices of degree $3$. The above proof shows that this happens if and only if $\calP$ is a quadrangulation of the sphere.

\subsection{Proof of Theorem \ref{thm:2}, $\Rightarrow$}
The assumptions are, $H$ is planar, of vertex connectivity $2$, such that $\calP=H\wedge K_2$ is a $3$-polytope.
We start with the following.

\begin{lemma}
	\label{lem:2cut}
	Let $\{a,b\}$ be a $2$-vertex-cut in a graph $H$, and $H\wedge K_2$ a $3$-polytope. Then the connected components of
	\[H-a-b\]
	are not bipartite. Moreover, there are exactly two components.
\end{lemma}
\begin{proof}
By contradiction, if a component of $H-a-b$ were bipartite, then one of
\begin{equation}
	\label{eq:abxy}
	\{ax,bx\}, \qquad \{ax,by\}, \qquad \{ay,bx\}, \qquad \{ay,by\}
\end{equation}
would form a $2$-cut in $H\wedge K_2$ (cf. Lemma \ref{lem:h-v}), impossible.

To prove the second statement, let $m\geq 2$ be the number of connected components in $H-a-b$, and $v_1,v_2,\dots,v_m$ vertices in distinct components. By the first part, the vertices
\[\{\{ax,ay,bx,by\},\{v_1x,v_2x,\dots,v_mx\}\}\]
determine a subgraph homeomorphic to $K(4,m)$ in $H\wedge K_2$. By Kuratowski's Theorem, $m=2$.
\end{proof}

By Lemma \ref{lem:2cut}, whenever $\{a,b\}$ is a $2$-cut in $H$, the graph $H-a-b$ has exactly two components, each of which contains an odd region. In particular, $H$ has at least two disjoint odd regions. This implication of Theorem \ref{thm:2} will be completely proven once we show the following.

\begin{lemma}
	\label{lem:ll}
	Let $H$ be a planar	graph such that $H\wedge K_2$ is a $3$-polytope. Suppose that $H$ has two disjoint odd regions. Then every odd region of $H$, save possibly two, contains a $2$-cut.
\end{lemma}
\begin{proof}
	Recall that if $H\wedge K_2$ is a $3$-polytope, then $H$ is $2$-connected. By contradiction, let $f_1,f_2$ be two disjoint odd regions, and $f_3\neq f_1,f_2$ an odd region such that none of $f_1,f_2,f_3$ contain a $2$-cut set of $H$. Then we claim that $H$ contains a subgraph that is an even subdivision of one of the graphs in Figure \ref{fig:ll}. Indeed, say that $f_3$ does not share any vertices with $f_1,f_2$. Then $H$ contains a subgraph that is an even subdivision of one of the graphs in Figures \ref{fig:l01}, \ref{fig:l02}, \ref{fig:l03}, and \ref{fig:l04}, depending on the existence of paths between $f_1,f_2$; $f_1,f_3$; $f_2,f_3$ of odd or even lengths. 
	\begin{figure}[h!]
		\centering
		\begin{subfigure}{0.19\textwidth}
			\centering
			\includegraphics[width=\widone]{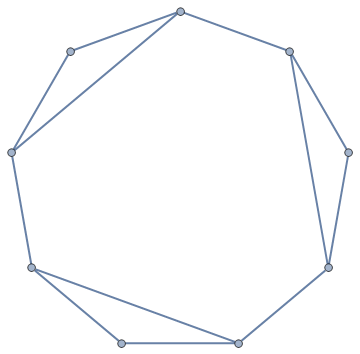}
			\caption{}
			\label{fig:l01}
		\end{subfigure}
		\hfill
		\begin{subfigure}{0.19\textwidth}
			\centering\includegraphics[width=\widone]{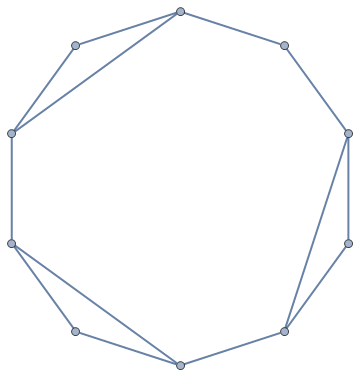}
			\caption{}
			\label{fig:l02}
		\end{subfigure}
		\hfill
		\begin{subfigure}{0.19\textwidth}
			\centering\includegraphics[width=\widone]{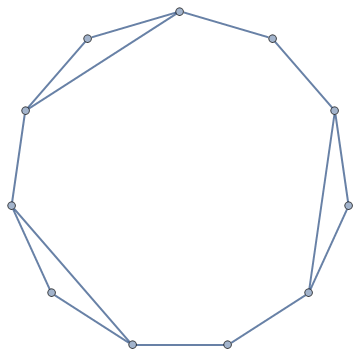}
			\caption{}
			\label{fig:l03}
		\end{subfigure}
		\hfill
		\begin{subfigure}{0.19\textwidth}
			\centering\includegraphics[width=\widone]{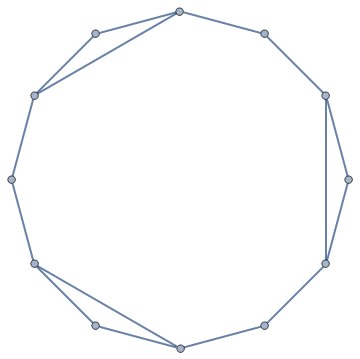}
			\caption{}
			\label{fig:l04}
		\end{subfigure}
		\hfill
		\begin{subfigure}{0.19\textwidth}
			\centering\includegraphics[width=\widone]{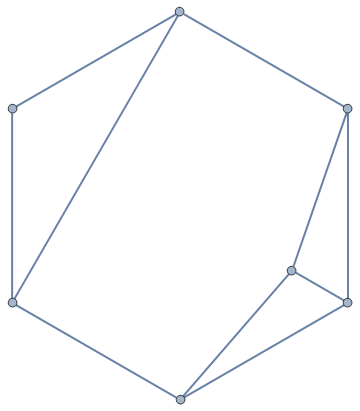}
			\caption{}
			\label{fig:l05}
		\end{subfigure}
		\hfill
		\begin{subfigure}{0.19\textwidth}
			\centering\includegraphics[width=\widone]{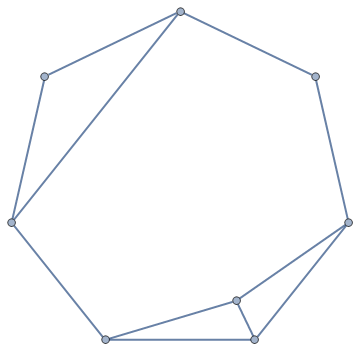}
			\caption{}
			\label{fig:l06}
		\end{subfigure}
		\hfill
		\begin{subfigure}{0.19\textwidth}
			\centering\includegraphics[width=\widone]{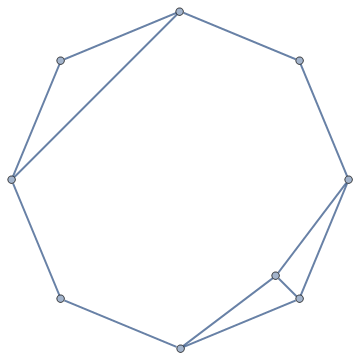}
			\caption{}
			\label{fig:l07}
		\end{subfigure}
		\hfill
		\begin{subfigure}{0.19\textwidth}
			\centering\includegraphics[width=\widone]{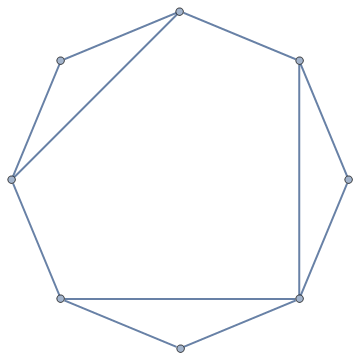}
			\caption{}
		\end{subfigure}
		\hfill
		\begin{subfigure}{0.19\textwidth}
			\centering\includegraphics[width=\widone]{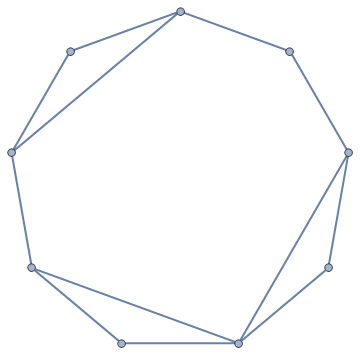}
			\caption{}
		\end{subfigure}
		\hfill
		\begin{subfigure}{0.19\textwidth}
			\centering\includegraphics[width=\widone]{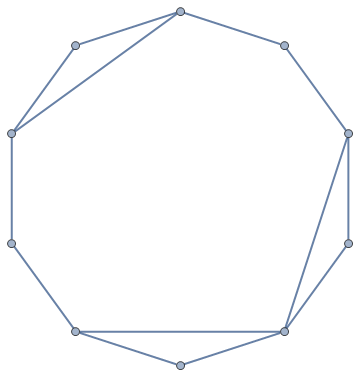}
			\caption{}
		\end{subfigure}
		\hfill
		\begin{subfigure}{0.19\textwidth}
			\centering\includegraphics[width=\widone]{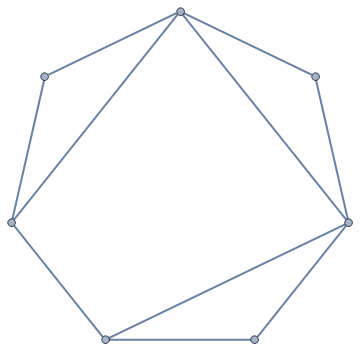}
			\caption{}
			\label{fig:l11}
		\end{subfigure}
		\hfill
		\begin{subfigure}{0.19\textwidth}
			\centering\includegraphics[width=\widone]{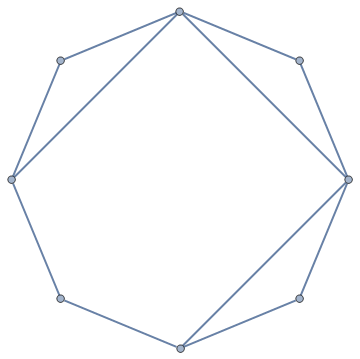}
			\caption{}
			\label{fig:l12}
		\end{subfigure}
		\hfill
		\begin{subfigure}{0.19\textwidth}
			\centering\includegraphics[width=\widone]{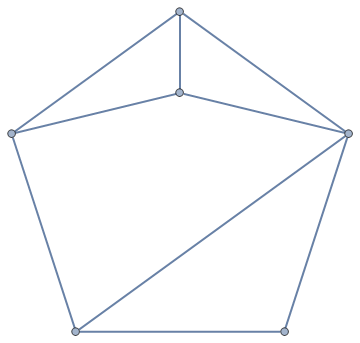}
			\caption{}
			\label{fig:l13}
		\end{subfigure}
		\hfill
		\begin{subfigure}{0.19\textwidth}
			\centering\includegraphics[width=\widone]{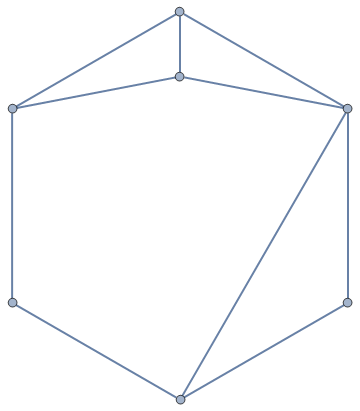}
			\caption{}
			\label{fig:l14}
		\end{subfigure}
		\hfill
		\begin{subfigure}{0.19\textwidth}
			\centering\includegraphics[width=\widone]{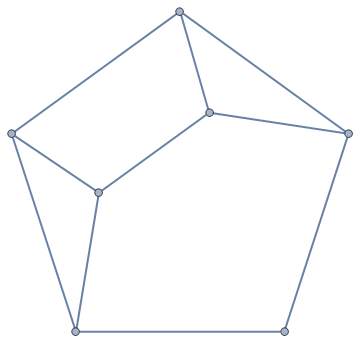}
			\caption{}
			\label{fig:l15}
		\end{subfigure}
		\caption{Lemma \ref{lem:ll}. Assume that $H\wedge K_2$ is a $3$-polytope, and that the planar $H$ has three odd regions with no $2$-cut,
		two of them disjoint. Then there exists a subgraph of $H$ that is an even subdivision of one of the depicted graphs.}
		\label{fig:ll}
	\end{figure}

All other cases (i.e., $f_3$ sharing vertices with $f_1,f_2$) are similarly depicted in Figures \ref{fig:l05}-\ref{fig:l15}. By Lemma \ref{lem:subd}, it suffices to check that each of the graphs in Figure \ref{fig:ll} has a non-planar Kronecker product with $K_2$, contradiction.
\end{proof}

The proof of the $\Rightarrow$ implication in Theorem \ref{thm:2} is complete.

\subsection{Proof of Theorem \ref{thm:2}, $\Leftarrow$}
\label{sec:thm2s}
We assume all of the following: $H$ is planar, of vertex connectivity $2$; all odd regions except exactly two contain a $2$-cut, and moreover if $\{a,b\}$
is a $2$-cut in $H$, then $H-a-b$
has exactly two connected components, each of them containing at least one odd region.

We will start with the special case of exactly two odd regions.
\begin{lemma}
	\label{lem:2odre}
	Let $H$ be a $2$-connected, plane graph with exactly two odd regions, which are disjoint. Suppose further that, if $\{a,b\}$ is a $2$-cut in $H$, then $H-a-b$
	has exactly two connected components, each of them containing an odd region. Then $H\wedge K_2$ is a $3$-polytope.
\end{lemma}
\begin{proof}
	Call $f,f'$ the two odd regions of $H$. We consider a list of distinct regions
	\begin{equation}
		\label{eq:list}
		f_1=f,f_2,\dots,f_{m},f_{m+1}=f', \quad m\geq 2,
	\end{equation}
	where consecutive regions in the list share an edge, and the length of the list is minimal, in the sense that non-consecutive regions in \eqref{eq:list} are not adjacent. We write $a_{i}b_{i}$ for the edge between $f_{i},f_{i+1}$, $1\leq i\leq m$. Let
	\[H':=H-a_1b_1-a_2b_2-\dots-a_{m}b_{m}.\]
	Up to relabelling, there is a region $r$ in $H'$ containing, among other vertices,
	\begin{equation}
		\label{eq:ab}
		a_1,a_2,\dots,a_{m},b_{m},b_{m-1},\dots,b_1
	\end{equation}
	in this order (Figure \ref{fig:odreo}). Some of these vertices may coincide, however, note that $a_1,b_1,a_{m},b_{m}$ are all distinct since $f,f'$ are disjoint.
	\begin{figure}
		\centering
		\includegraphics[width=7.25cm]{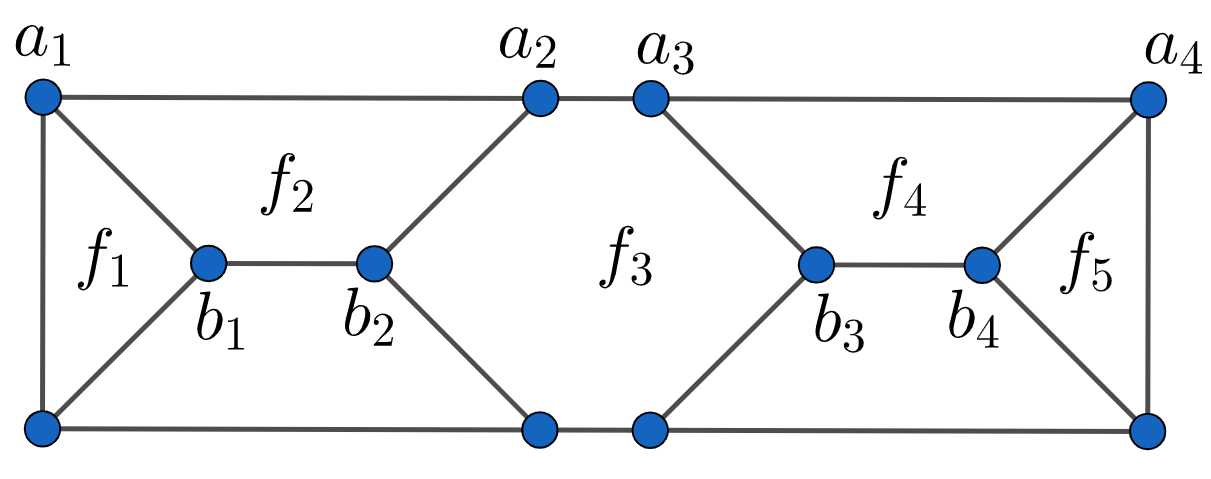}
		\hspace{0.25cm}
		\includegraphics[width=7.25cm]{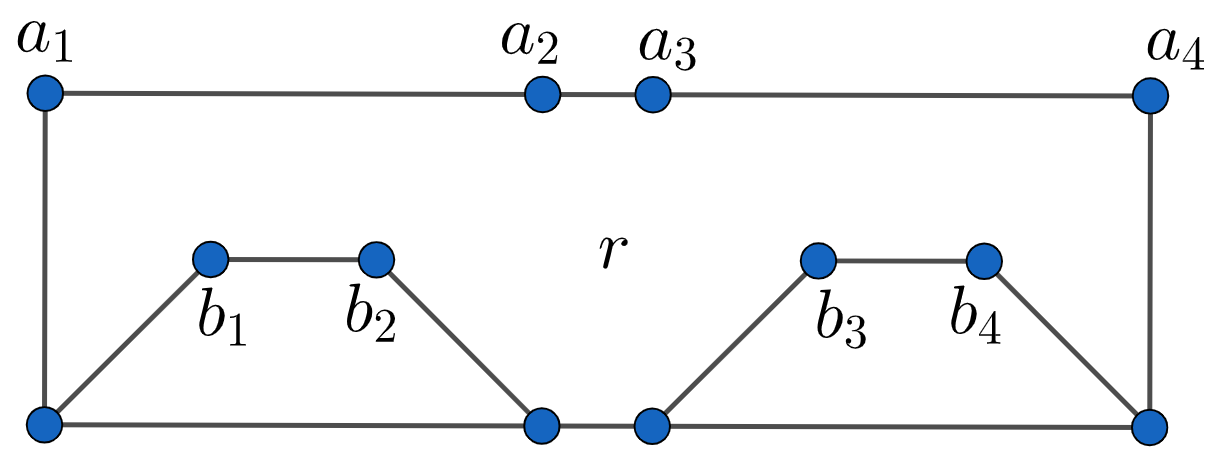}
		\caption{A graph $H$ satisfying the assumptions of Theorem \ref{thm:2}, with exactly two odd regions (left), and a possible subgraph $H'$ (right).}
		\label{fig:odreo}
	\end{figure}
	
	Let us check that $H'$ is $2$-connected. Equivalently, we check that each region of $H'$ is bounded by a cycle. By construction, it suffices to show that $r$ is bounded by a cycle, and in particular that no vertex on its boundary has degree $1$ in $H'$. By contradiction, let
	\[\deg_{H'}(a_i)=1,\]
	for some $i=1,\dots,m$. Since $\delta(H)\geq 3$, then $a_ib_j$, $a_ib_{j+1}$, \dots, $a_ib_{j+k}$ are edges in $H$, for some $k\geq 1$ and $1\leq j\leq m-k$. Thereby, $f_{j}$ and $f_{j+k+1}$ are in fact adjacent regions in $H$, contradicting the definition of \eqref{eq:list}.
	
	By construction, the graph $H'$ is bipartite, as all of its regions are even. We sketch two copies of it in the plane, with $r$ the external region in both. We then add the edges
	\[(a_i,x)(b_iy), \quad (a_i,y)(b_ix), \quad 1\leq i\leq m.\]
	It is clear that the resulting $H\wedge K_2$ is still planar.
	
	It remains to show $3$-connectivity for $H\wedge K_2$. As remarked above, $a_1,b_1,a_{m},b_{m}$ are distinct vertices. Moreover, in the above construction of $H\wedge K_2$, the vertices
	\[(a_1,x),(b_1,x)\]
	(resp. $(a_{m},x),(b_{m},x)$) lie on the same copy of $H'$, since $f$ (resp. $f'$) is an odd region. Hence there are at least four distinct edges \[(a_1,x)(b_1y), \quad (a_1,y)(b_1x), \quad (a_{m},x)(b_{m}y), \quad (a_{m},y)(b_{m}x)\]
	connecting the two copies of $H'$. Lastly, let $\{a,b\}$ be any $2$-cut in $H$. By assumption, $H-a-b$
	has exactly two components, each of them containing an odd region. Hence $a,b$ lie on $r$, but not between $a_1,b_1$ or $a_m,b_m$, in the sense that, if
	\[v_1,v_2,\dots,v_V,a_1,a_2,\dots,a_m,w_1,w_2,\dots,w_W,b_m,b_{m-1},\dots,b_1\]
	lie on $r$ in this order, then $\{a,b\}$ is not a subset of
	\[\text{either } \{b_1,v_1,v_2,\dots,v_{V},a_1\} \ \text{ or }\{a_m,w_1,w_2,\dots,w_{W},b_m\}.\]
	Thereby, $H\wedge K_2$ is $3$-connected.
\end{proof}

We will adapt the proof of Lemma \ref{lem:2odre} to the more general case. We consider a planar embedding of $H$. Let $f,f'$ be the two odd regions with no $2$-cut. One takes a list of regions
\begin{equation}
	\label{eq:listgen}
	f_1=f,f_2,\dots,f_{m},f_{m+1}=f', \quad m\geq 2,
\end{equation}
where consecutive regions in the list share an edge, and the length of \eqref{eq:listgen} is minimal, in the sense that non-consecutive regions in \eqref{eq:listgen} are not adjacent. By assumption, all odd regions
of $H$ appear in \eqref{eq:listgen}.

We can label the edges shared by consecutive elements of \eqref{eq:listgen} so that $a_{i}b_{i}$ is the edge between $f_{i},f_{i+1}$, $1\leq i\leq m$, and moreover,
letting
\[H':=H-a_1b_1-a_2b_2-\dots-a_{m}b_{m},\]
there is a region $r$ in $H'$ containing, among other vertices,
\[a_1,a_2,\dots,a_{m},b_{m},b_{m-1},\dots,b_1\]
in this order (Figure \ref{fig:odret}).
\begin{figure}
	\centering
	\includegraphics[width=7.25cm]{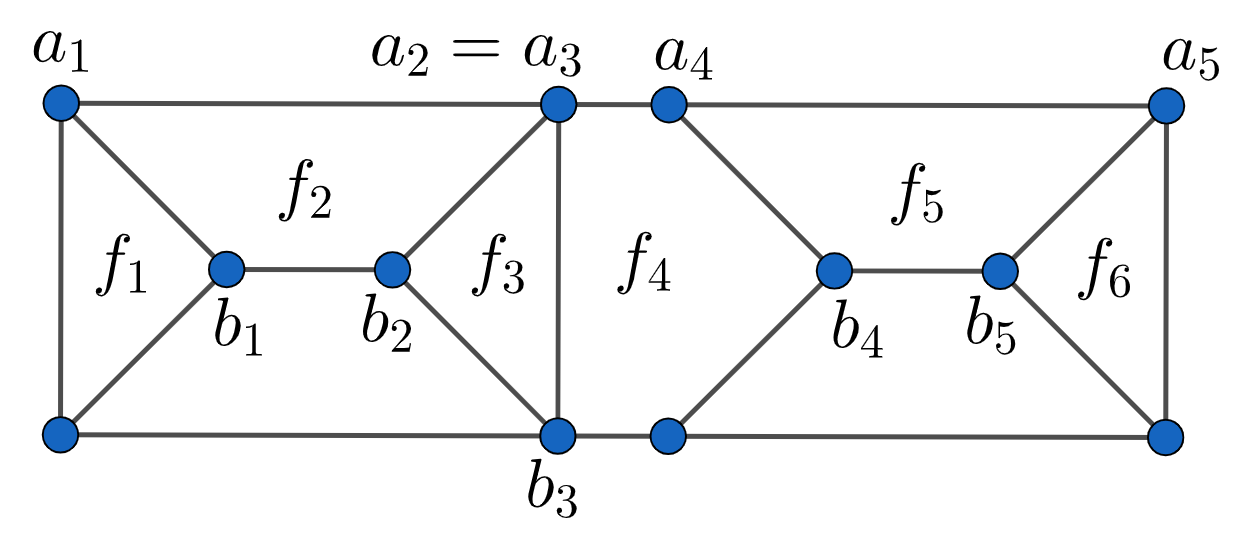}
	\hspace{0.25cm}
	\includegraphics[width=7.25cm]{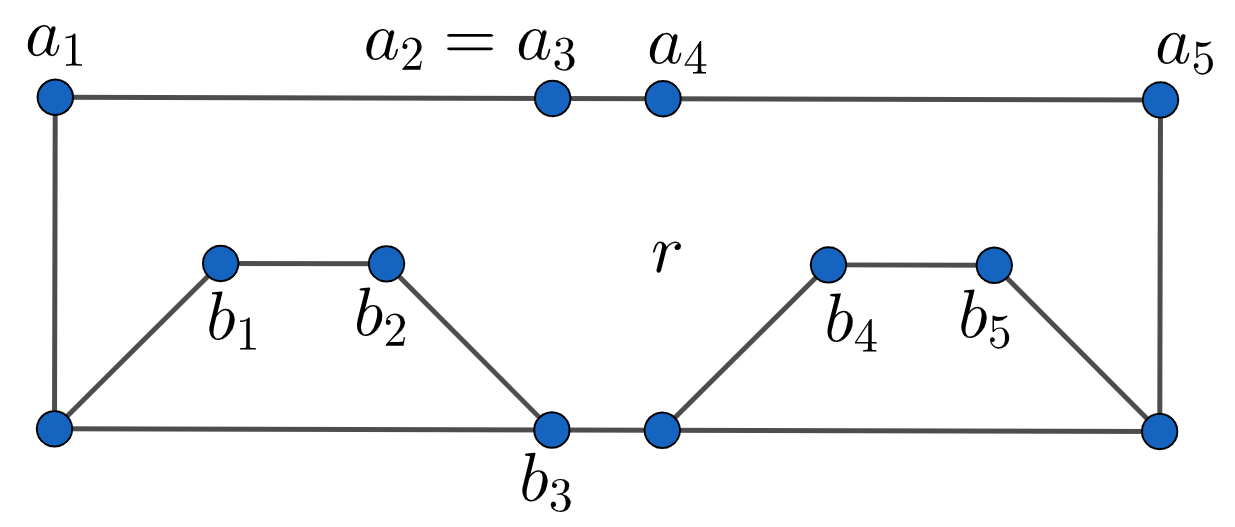}
	\caption{The graph $H$ from Figure \ref{fig:pl2c}, and a possible subgraph $H'$.}
	\label{fig:odret}
\end{figure}

The new graph $H'$ is planar and $2$-connected (cf. the proof of Lemma \ref{lem:2odre}). Moreover, it is bipartite, since all odd regions of $H$ appear in the list \eqref{eq:listgen}. We sketch two copies of $H'$ in the plane, with $r$ the external region in both. We then add the edges
\[(a_i,x)(b_i,y), \quad (a_i,y)(b_i,x), \qquad 1\leq j\leq m.\]
As argued in the proof of Lemma \ref{lem:2odre}, the resulting $H\wedge K_2$ is planar and $3$-connected.

\section{Proof of Theorem \ref{thm:1}, $\Rightarrow$}
\label{sec:thm1n}
\subsection{First part}
Let $H$ and $\calP=H\wedge K_2$ be $3$-polytopes. Recall that $O$ denotes the set of odd faces in $H$. Clearly $|O|$ is even (handshaking lemma) and $|O|\neq 0$, otherwise $H$ would be bipartite, thus $\calP$ would be disconnected due to Lemma \ref{lem:h-v}. If $|O|=2$, then by Lemma \ref{lem:h-v}, the two odd faces of $H$ are disjoint. This is Condition \ref{eq:c1} in Theorem \ref{thm:1}. On the other hand, let us show that, if there are more than two odd faces, any two of them are not disjoint. For the rest of this section, we will be assuming $|O|\geq 4$.

\begin{cor}
	\label{lem:notdisj}
	Let $H$ and $\calP=H\wedge K_2$ be $3$-polytopes, and $O$ the set of odd faces in $H$. If $|O|\geq 4$, then any two odd faces of $H$ are not disjoint.
\end{cor}
\begin{proof}
We apply Lemma \ref{lem:ll}: since $H$ is $3$-connected, if it contains two disjoint odd faces, then these are the only two odd faces.
\end{proof}

We define the subgraph of $H$ generated by the edges lying on odd faces,
\begin{equation}
	\label{eq:ho}
\ho=\langle\{\text{edges of } H \text{ lying on odd faces}\}\rangle.
\end{equation}
Two examples are shown in Figure \ref{fig:ho12}.
\begin{figure}[h!]
	\centering
	\begin{subfigure}{0.49\textwidth}
		\centering
		\includegraphics[width=3.5cm]{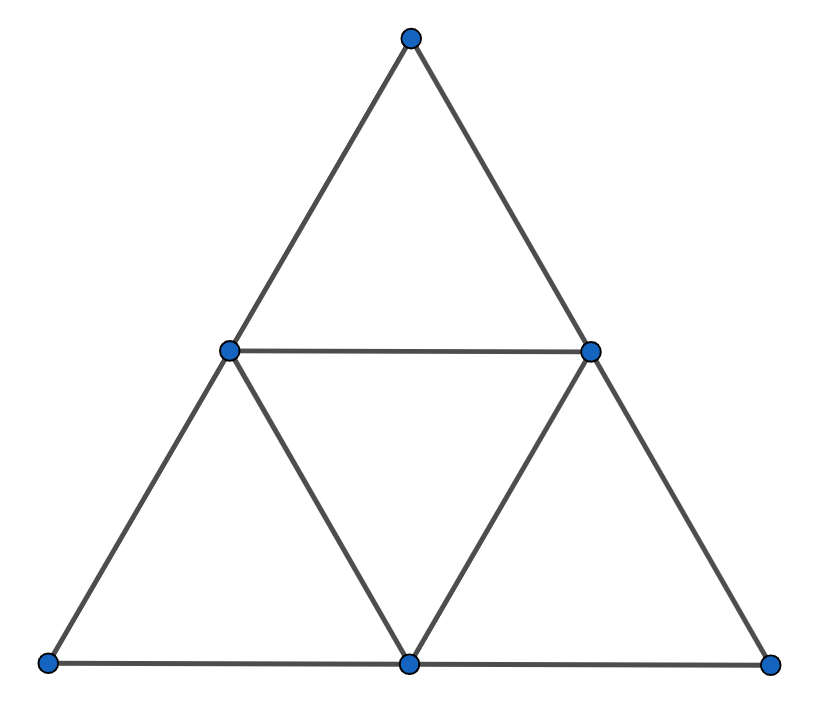}
		\caption{The subgraph $\ho$ of $H$ in Figure \ref{fig:c3}.}
		\label{fig:ho1}
	\end{subfigure}
	\begin{subfigure}{0.49\textwidth}
		\centering
		\includegraphics[width=3.5cm]{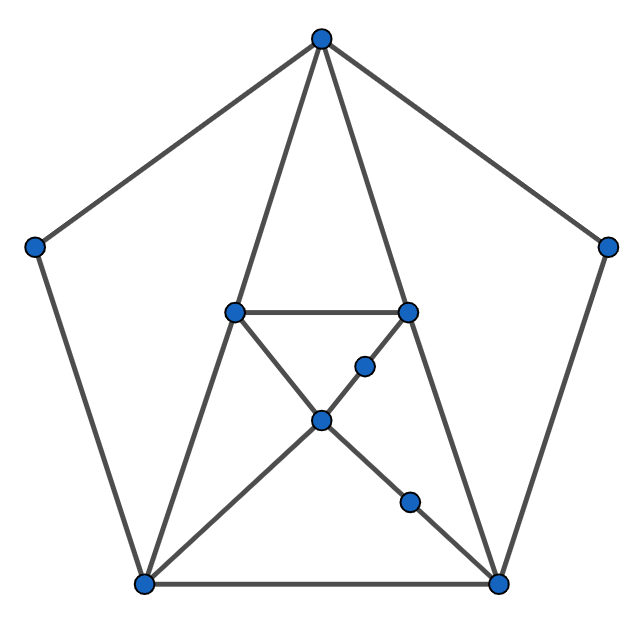}
		\caption{The subgraph $\ho$ of $H$ in Figure \ref{fig:c2}.}
		\label{fig:ho2}
	\end{subfigure}
	\caption{Two examples of $\ho$.}
	\label{fig:ho12}
\end{figure}
 	
Corollary \ref{lem:notdisj} implies that, if $|O|\geq 4$, then $\ho$ is connected. Using this fact, we can now show that it is in fact $2$-connected.
\begin{cor}
	\label{cor:ho}
	Let $H$ and $\calP=H\wedge K_2$ be $3$-polytopes, and $O$ the set of odd faces in $H$. If $|O|\geq 4$, then
	$\ho$ is $2$-connected.
\end{cor}
\begin{proof}
	By contradiction, let $\ho$ be connected but not $2$-connected, and $u$ a separating vertex. By Corollary \ref{lem:notdisj}, all odd faces of $H$ contain $u$. Hence $H-u$ is bipartite, contradicting Lemma \ref{lem:h-v}.
\end{proof}

We can now prove that, if \emph{all} faces of $H$ are odd, then $H$ is simply an odd pyramid. In particular, Condition \ref{eq:c3} in Theorem \ref{thm:1} is met.
\begin{lemma}
Let $H$ and $\calP=H\wedge K_2$ be $3$-polytopes. If all faces of $H$ are odd, then $H$ is an $n$-gonal pyramid for some odd $n\geq 3$.
\end{lemma}
\begin{proof}
	We consider two cases. First case, assume that $H$ is not a triangulation, and fix a face
	\[f=[u_1,u_2,\dots,u_n], \quad n>3.\]
	Consider four distinct vertices
	\[u_i,u_{(i+1 \bmod n)},u_{(i+2 \bmod n)},u_{(i+3 \bmod n)}\]
	(consecutive on $f$) and call $f_{i}'$ the face sharing $u_i,u_{(i+1 \bmod n)}$ with $f$, and $f_{i}''$ the face sharing $u_{(i+2 \bmod n)},u_{(i+3 \bmod n)}$ with $f$. By assumption, all faces are odd, so that by Corollary \ref{lem:notdisj}, $f_i',f_i''$ share either a vertex or edge. This arguments holds for every $1\leq i\leq n$, so that by planarity of $H$, all faces
	adjacent to $f$ share a common vertex. That is to say, $H$ is the $n$-gonal pyramid for some odd $n\geq 5$.
	
	Second case, let $H$ be a triangulation. It cannot have any vertex of even degree, otherwise it would contain an even subdivision of the square pyramid, contradicting Corollary \ref{cor:4pyr}. Let $\deg_H(u)=m\geq 5$. Its neighbours form an $m$-cycle
	\[u_1,u_2,\dots,u_m.\]
	By Corollary \ref{lem:notdisj}, the triangular face sharing $u_1u_2$ with the face
	\[[u,u_1,u_2]\]
	has non-empty intersection with both of
	\[[u,u_3,u_4], \qquad [u,u_4,u_5],\]
	so that its third vertex is $u_4$. Therefore, $H$ contains a square pyramid, contradicting Corollary \ref{cor:4pyr}. It follows that each vertex of the triangulation $H$ is of degree $3$, i.e., $H$ is the tetrahedron.
\end{proof}

\subsection{Second part}
So far, we have dealt with the cases when $|O|=2$ and when $H$ contains only odd faces. Henceforth, we will assume that $|O|\geq 4$ and  that $H$ contains at least one even face.
\begin{lemma}
	\label{lem:23}
Let $H$ and $\calP=H\wedge K_2$ be $3$-polytopes, and $O$ the set of odd faces in $H$. If $|O|\geq 4$ and $H$ contains an even face, then either \ref{eq:c2} or \ref{eq:c3} from Theorem \ref{thm:1} holds.
\end{lemma}
\begin{proof}
Recalling \eqref{eq:ho} and Corollary \ref{cor:ho}, $\ho$ is the planar, $2$-connected subgraph of $H$ generated by edges lying on odd faces. Let
\[r=[a_1,a_2,\dots,a_n]\]
be an even region of $\ho$. By definition of $\ho$, $r$ is adjacent in $\ho$ only to odd regions.

We claim that $r$ is adjacent to exactly three odd regions in $\ho$. Firstly, it is clearly adjacent to at least two distinct odd regions. If exactly two, say $r_1,r_2$, then the two corresponding odd faces of $H$ share two vertices $a_i,a_j$ with $1\leq i<j\leq n$. By $3$-connectivity of $H$, we deduce that $r_1,r_2$ are adjacent. In turn, this means that either $\ho$ is not $2$-connected, or there are exactly two odd faces in $H$. We have reached a contradiction. Now assume that $r$ is adjacent to four or more odd regions in $\ho$. By Corollary \ref{lem:notdisj}, no two odd regions are disjoint. Therefore, by planarity, all odd regions adjacent to $r$ share a vertex. Since $\ho$ is $2$-connected, this vertex lies on all odd faces of $H$, contradicting Lemma \ref{lem:h-v}. We have succeeded in proving that $r$ is adjacent to exactly three odd regions in $\ho$. We will label these $f_1,f_2,f_3$. We write
\[V(f_1\cap f_2)=\{a_i,b_i\}, \quad V(f_1\cap f_3)=\{a_j,b_j\}, \quad V(f_2\cap f_3)=\{a_k,b_k\},\]
where $a_i,a_j,a_k$ are three distinct vertices on $r$ as defined above, and possibly any number of $a_i=b_i$, $a_j=b_j$, $a_k=b_k$ hold (Figure \ref{fig:hogen}).
\begin{figure}[h!]
	\centering
	\includegraphics[width=6.5cm]{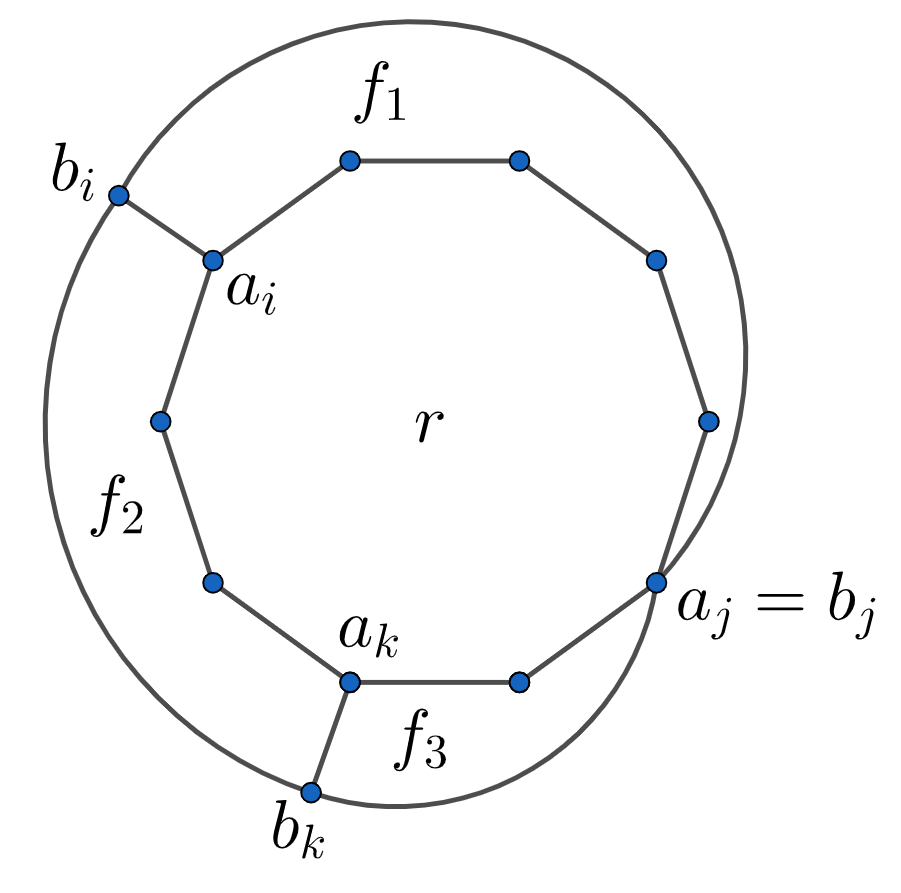}
	\caption{Illustration of the construction in Lemma \ref{lem:23}.}
\label{fig:hogen}
\end{figure}

By Corollary \ref{lem:notdisj}, all other odd regions of $\ho$ intersect each of $f_1,f_2,f_3$ non-trivially. We distinguish between three mutually exclusive scenarios for the odd faces of $H$.
\begin{itemize}
	\item 
	There exists in $H$ an odd face containing all of $b_i,b_j,b_k$. Hence $b_ib_j,b_ib_k,b_jb_k\in E(H)$, so that $f_4$ is a triangle. Hence there are exactly four odd faces in total, namely $f_4$ and the three faces adjacent to it. Condition \ref{eq:c3} in Theorem \ref{thm:1} is satisfied.
	\item 
	There exists an odd face $f_4$ in $H$ containing none of $b_i,b_j,b_k$. By planarity, $f_1,f_2,f_3,f_4$ are the only odd faces. Each pair of odd faces has non-empty intersection, and each triple of odd faces has empty intersection. This is Condition \ref{eq:c2} in Theorem \ref{thm:1}.
	\item
	The remaining case is, each odd face of $H$ contains one or two of $b_i,b_j,b_k$. By planarity of $H$, Condition \ref{eq:c3} in Theorem \ref{thm:1} is met unless there exist odd faces $f_4,f_5,f_6$ (distinct from  $f_1,f_2,f_3$) not containing, respectively, $b_k,b_j,b_i$. By planarity, this is possible only if $f_1,f_4$ share the edge $b_ib_j$, $f_2,f_5$ share $b_ib_k$, and $f_3,f_6$ share $b_jb_k$. The $\Rightarrow$ statement of Theorem \ref{thm:1} will be completely proven once we exclude this possibility in the following Lemma \ref{lem:ijk}.
\end{itemize}
\end{proof}

\begin{lemma}
	\label{lem:ijk}
Let $H$ be a $3$-polytope, $1\leq i<j<k$ integers, and $b_ib_j$, $b_ib_k$, $b_jb_k$ edges that are each shared between two odd faces. Then $H\wedge K_2$ is non-planar.
\end{lemma}
\begin{proof}
We will prove that $H$ contains an even subdivision of the square pyramid, to apply Corollary \ref{cor:4pyr}. We claim that there are exactly six odd faces in $H$, namely, the faces containing one of $b_ib_j$, $b_ib_k$, $b_jb_k$. Indeed, by Corollary \ref{lem:notdisj}, another odd face $f_7$ of $H$ would have non-empty intersection with each of these six. By planarity of $H$, $f_7$ would contain at least two of $b_i,b_j,b_k$. This would contradict the $3$-connectivity of $H$.

Similarly to Lemma \ref{lem:23}, we can find an even region of $\ho$
\[r=[a_1,a_2,\dots,a_n], \quad n\geq k,\]
such that
\[\text{dist}_H(a_i,b_i), \ \text{dist}_H(a_j,b_j), \ \text{dist}_H(a_k,b_k)\leq 1.\]
That is to say, there exist $a_{i'},a_{j'},a_{k'}\in r$ adjacent to $b_i,b_j,b_k$ respectively, and satisfying
\[|i'-i|, |j'-j|, |k'-k|\leq 1.\]
Up to relabelling, we can assume that $|i'-k'|\geq 2$, since $n\geq 4$, and write
\begin{align*}
f_1&=[b_i,b_j,\dots,a_{i'}], \qquad f_2=[b_k,b_i,\dots,a_{k'}],\\ f_4&=[b_i,b_j,\dots,c_4], \qquad
f_5=[b_i,b_k,\dots,c_5],\\
\end{align*}
where possibly $c_4=c_5$ (Figure \ref{fig:ijk}).
\begin{figure}[h!]
	\centering
	\includegraphics[width=6.5cm]{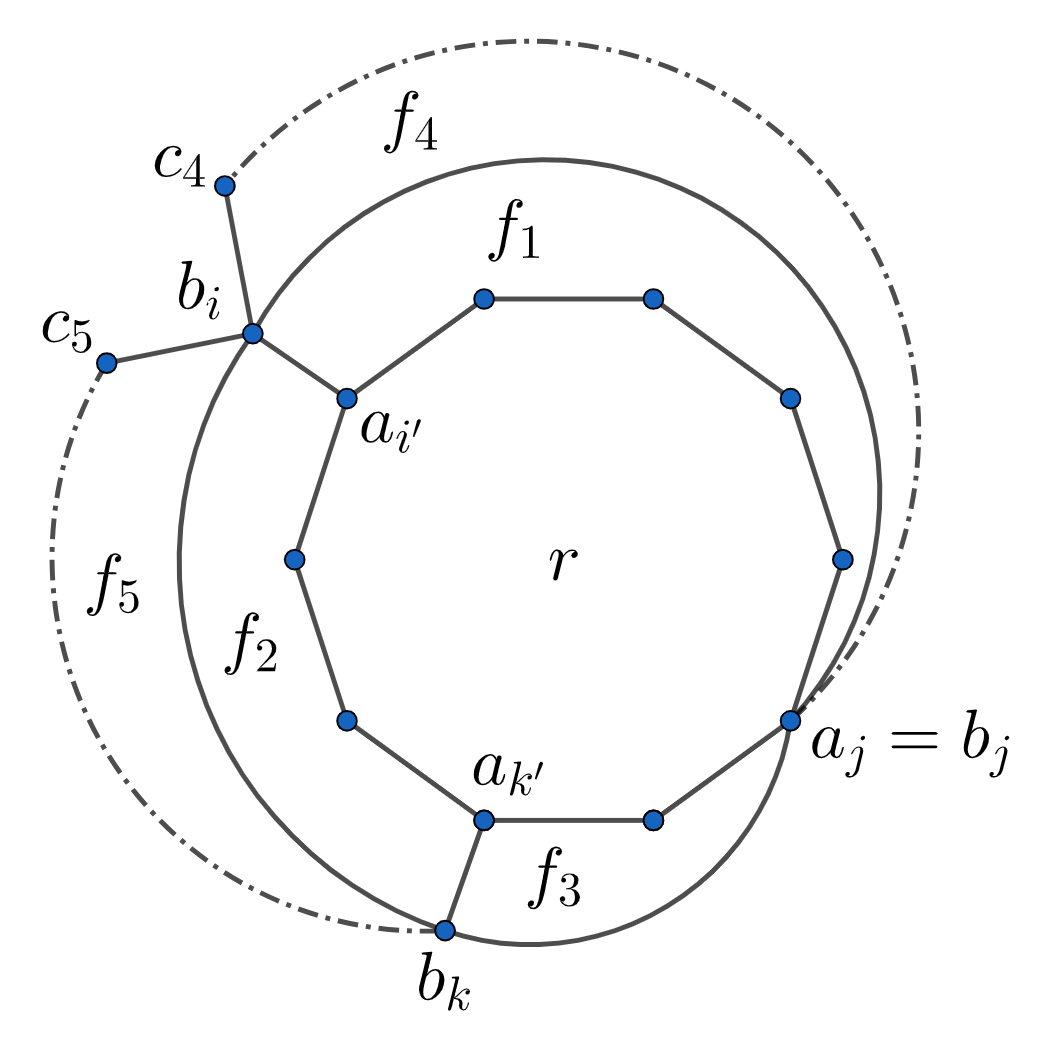}
	\caption{Construction of Lemma \ref{lem:ijk}. The dash-dotted lines represent (even) paths on at least two vertices.}
	\label{fig:ijk}
\end{figure}

By construction, $b_i$ is adjacent to each of
\[a_{i'},b_j,c_4,b_k.\]
There is an odd $a_{i'}b_j$-path $\psi_1$ along $f_1$, and an odd $b_jc_4$-path $\psi_2$ along $f_4$. By $3$-connectivity of $H$, there is an even $a_{i'}a_{k'}$-path not containing any other vertex of $f_1$. Since $a_{k'}$ and $b_k$ are adjacent, we hence find an odd $a_{i'}b_k$-path $\psi_3$ that is internally disjoint from $\psi_1$, $\psi_2$. Now there are two cases. If $c_4=c_5$, we simply have an odd $c_4b_k$-path $\psi_4$ along $f_5$, hence this path is internally disjoint from $\psi_1$, $\psi_2$, $\psi_3$. We have been successful in finding an even subdivision of the square pyramid within $H$.

Now assume instead that $c_4\neq c_5$. Recall that there are exactly six odd faces in $H$. By $3$-connectivity there is an even $c_4c_5$-path $\psi_4'$ not containing any other vertices of $f_4$ or $f_5$. In any case, $H$ contains an even subdivision of the square pyramid.
\end{proof}

\section{Proof of Theorem \ref{thm:1}, $\Leftarrow$}
\label{sec:thm1s}
We need to show that if a $3$-polytope $H$ satisfies any of the three conditions listed in Theorem \ref{thm:1}, then $H\wedge K_2$ is also a $3$-polytope. For Condition \ref{eq:c1}, we simply recall Lemma \ref{lem:2odre}.

Let us show that Condition \ref{eq:c2} of Theorem \ref{thm:1} is sufficient for $H\wedge K_2$ to be a $3$-polytope.
\begin{lemma}
	\label{lem:c2}
	Let $H$ be a $3$-polytope and $O$ the set of its odd faces, satisfying $|O|=4$. Suppose that any two elements of $O$ intersect in a vertex or edge, while any three elements of $O$ have empty intersection. Then $\calP=H\wedge K_2$ is a $3$-polytope.
\end{lemma}
\begin{proof}
Recall that $V(K_2)=\{x,y\}$. Let $f_1,f_2,f_3,f_4$ be the odd faces of $H$. We distinguish between three cases.

{\bf First case.} The faces $f_1,f_2$ share an edge $ab$ and the faces $f_3,f_4$ share an edge $cd$, as in Figure \ref{fig:abcd1}. We define the planar, $2$-connected, bipartite graph
\[H'=H-ab-cd.\]
We write
\[H\wedge K_2=H'\dot\cup H'+(a,x)(b,y)+(a,y)(b,x)+(c,x)(d,y)+(c,y)(d,x).\]
By construction, this graph is $3$-connected. It now suffices to point out that there is an immersion of $H'$ in the plane such that $a,c,b,d$ appear in this order along the boundary of a region $r$ -- Figure \ref{fig:abcd2}. It remains to sketch two copies of $H'$, with $r$ the external region in both, to see that $H\wedge K_2$ is planar.
\begin{figure}[h!]
	\centering
	\begin{subfigure}{0.45\textwidth}
		\centering
		\includegraphics[width=5.5cm]{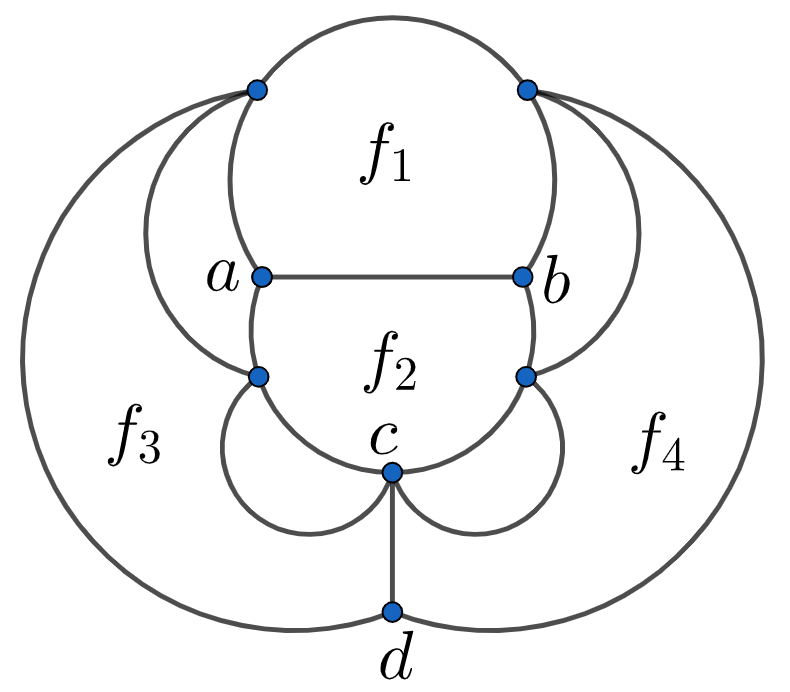}
		\caption{In $H$, the odd faces $f_1,f_2$ share an edge $ab$ and the odd faces $f_3,f_4$ share an edge $cd$.}
		\label{fig:abcd1}
	\end{subfigure}
\hspace{0.75cm}
	\begin{subfigure}{0.45\textwidth}
		\centering
		\includegraphics[width=5.5cm]{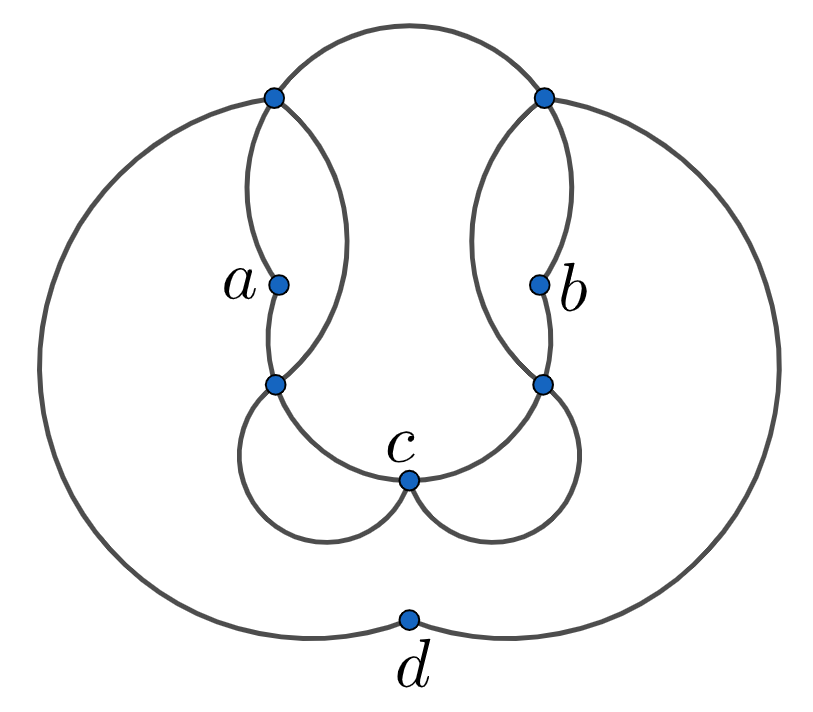}
		\caption{After deleting $ab$ and $cd$, there exists a planar immersion of the resulting graph such that $a,c,b,d$ lie on a region in this order.}
		\label{fig:abcd2}
	\end{subfigure}
	\caption{Lemma \ref{lem:c2}, First case.}
	\label{fig:abcd}
\end{figure}

{\bf Second case.}
The faces $f_1,f_2$ share an edge $ab$ and the faces $f_3,f_4$ share only the vertex $v$. The graph
\[H'=H-ab-v\]
is $2$-connected (as well as planar and bipartite), since $ab$ and $v$ do not lie on the same face of $H$ by assumption. Let
\[v_1,v_2,\dots,v_n, \quad n\geq 4,\]
be the neighbours of $v$ in $H$, labelled so that they appear in this order along the boundary of a region $r$ in $H'$, and so that
\[v_1,v_n\in f_3, \quad v_i,v_{i+1}\in f_4, \qquad \text{ for some fixed }2\leq i\leq n-2\]
(Figure \ref{fig:abu1}). Similarly to the previous case, we may immerse $H'$ in the plane so that $a,b$ also belong to $r$. To be precise,
\begin{equation}
	\label{eq:abu}
a,v_1,v_2,\dots,v_i,b,v_{i+1},v_{i+2},\dots,v_n
\end{equation}
lie, in this order, around the boundary of $r$, as in Figure \ref{fig:abu2}.
\begin{figure}[h!]
	\centering
	\begin{subfigure}{0.45\textwidth}
		\centering
		\includegraphics[width=5.5cm]{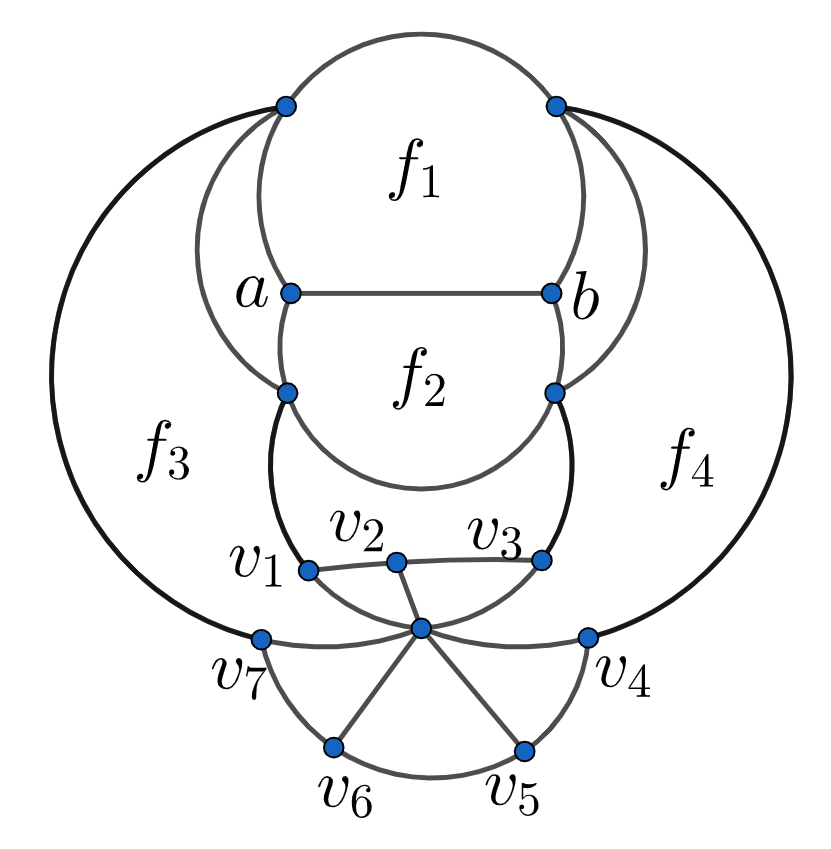}
		\caption{In $H$, the odd faces $f_1,f_2$ share an edge $ab$ and the odd faces $f_3,f_4$ share the vertex $v$.}
		\label{fig:abu1}
	\end{subfigure}
\hspace{0.75cm}
	\begin{subfigure}{0.45\textwidth}
		\centering
		\includegraphics[width=5.5cm]{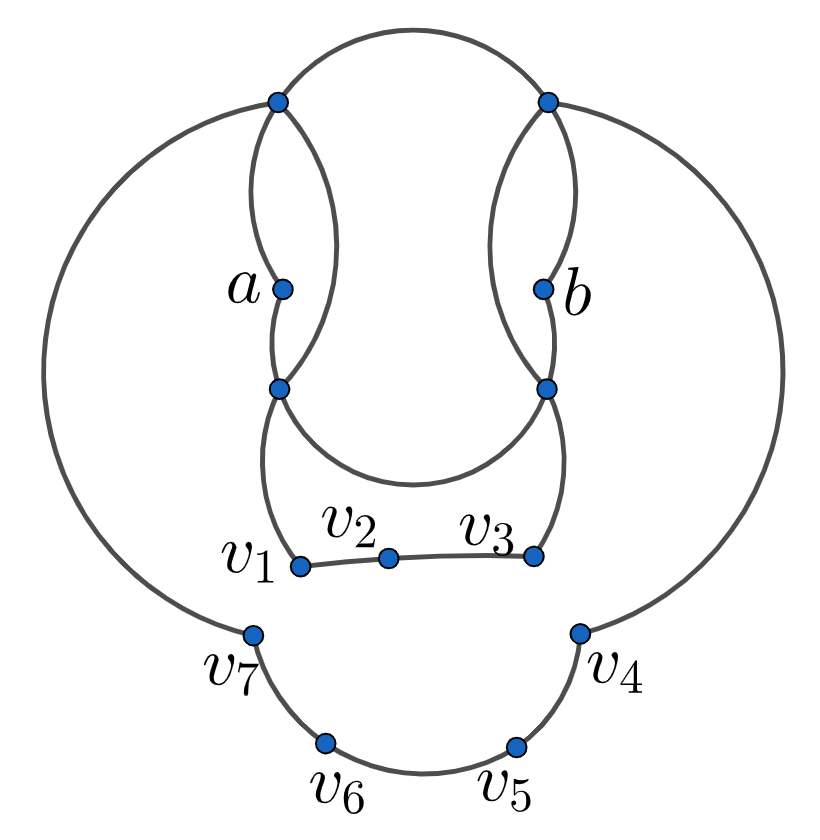}
		\caption{After deleting $ab$ and $v$, there exists a planar immersion of the resulting graph such that \eqref{eq:abu} lie on a region in this order (here $n=7$ and $i=3$).}
		\label{fig:abu2}
	\end{subfigure}
	\caption{Lemma \ref{lem:c2}, Second case.}
	\label{fig:abu}
\end{figure}

We note that, in a $2$-colouring of $V(H')$, $v_1$ and $v_n$ have opposite colours, since in $H'$ there exists a $v_1v_n$-path of odd length, formed by all the edges of $f_3$ save $vv_1$ and $vv_n$. Via similar considerations, we see that $v_1,v_2,\dots,v_i$ are of one colour and $v_{i+1},v_{i+2},\dots,v_n$ of the other, and moreover $a,b$ are of the same colour.

To draw $H\wedge K_2$, we start with two copies of $H'$, with labelling such that, in one copy, w.l.o.g.
\[ax,v_1x,v_2x,\dots,v_ix,bx,v_{i+1}y,v_{i+2}y,\dots,v_ny\]
appear in this order around the contour of $r$, and in the other copy,
\[ay,v_1y,v_2y,\dots,v_iy,by,v_{i+1}x,v_{i+2}x,\dots,v_nx\]
lie in this order around the boundary of $r$. To pass to $H\wedge K_2$, we add the edges
\[(a,x)(b,y), \ (a,y)(b,x), \ (v,x)(v_i,y), \ (v,y)(v_i,x), \qquad 1\leq i\leq n.\]
The resulting graph is a $3$-polytope.

{\bf Third case.}
The faces $f_1,f_2$ share only the vertex $u$ and the faces $f_3,f_4$ share only the vertex $v$. With similar ideas to the previous cases,
\[H'=H-u-v\]
is planar, $2$-connected and bipartite. The neighbours of $u$
\[u_1,u_2,\dots,u_m, \quad m\geq 4,\]
and those of $v$
\[v_1,v_2,\dots,v_n, \quad n\geq 4\]
may be labelled so that, in a $2$-colouring of $V(H)$,
\[u_1,u_2,\dots,u_i,v_1,v_2,\dots,v_j\]
are of the same colour, and
\[u_{i+1},u_{i+2},\dots,u_m,v_{j+1},v_{j+2},\dots,v_n\]
of the other colour, for some fixed $2\leq i\leq m-2$ and $2\leq j\leq n-2$. We sketch two copies of $H'$, where in one copy the external region contains
\[u_1x,u_2x,\dots,u_ix,v_1x,v_2x,\dots,v_jx,u_{i+1}y,u_{i+2}y,\dots,u_my,v_{j+1}y,v_{j+2}y,\dots,v_ny\]
in this order, and in the other copy the external region contains
\[u_1y,u_2y,\dots,u_iy,v_1y,v_2y,\dots,v_jy,u_{i+1}x,u_{i+2}x,\dots,u_mx,v_{j+1}x,v_{j+2}x,\dots,v_nx\]
in this order. Thus we check that $H\wedge K_2$ is a $3$-polytope.
\end{proof}

It remains to prove that Condition \ref{eq:c3} of Theorem \ref{thm:1} is sufficient for $H\wedge K_2$ to be a $3$-polytope.
\begin{lemma}
	Let $H$ be a $3$-polytope with at least four odd faces. If all odd faces except one share a common vertex $u$, and the remaining odd face $f$ has non-empty intersection with all other odd faces, then $\calP=H\wedge K_2$ is a $3$-polytope.
\end{lemma}
\begin{proof}
Call
\[f_1,f_2,\dots,f_{2D-1}, \quad D\geq 2,\]
the odd faces of $H$ containing the vertex $u$. The graph
\[H'=H-u\]
is planar and $2$-connected. We denote by $r$ the only region of $H'$ that is not a face of $H$. Note that $H'$ has exactly two odd regions, namely $r$ and $f$. For an illustration of $H,H'$, see Figures \ref{fig:aub1} and \ref{fig:aub2}.
\begin{figure}[h!]
	\centering
	\includegraphics[width=10cm]{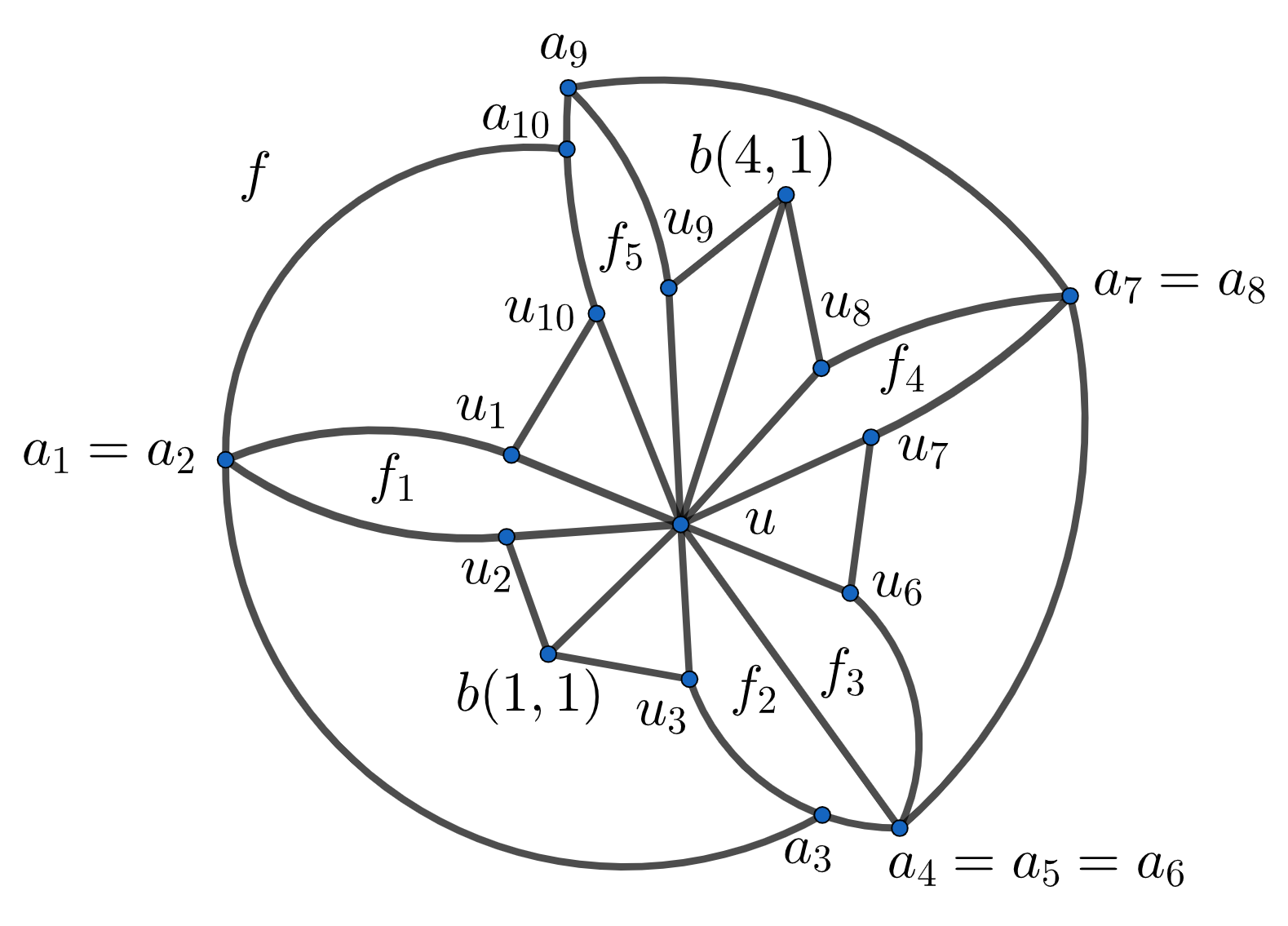}
	\caption{In this sketch for $H$, there are $2D=6$ odd regions. Also, $n_1=n_4=1$, and $n_2=n_3=n_5=0$. Note that $f_2,f_3$ are adjacent faces, hence $a_4=u_4=u_5=a_5$.}
	\label{fig:aub1}
\end{figure}

\begin{figure}[h!]
	\centering
	\includegraphics[width=10cm]{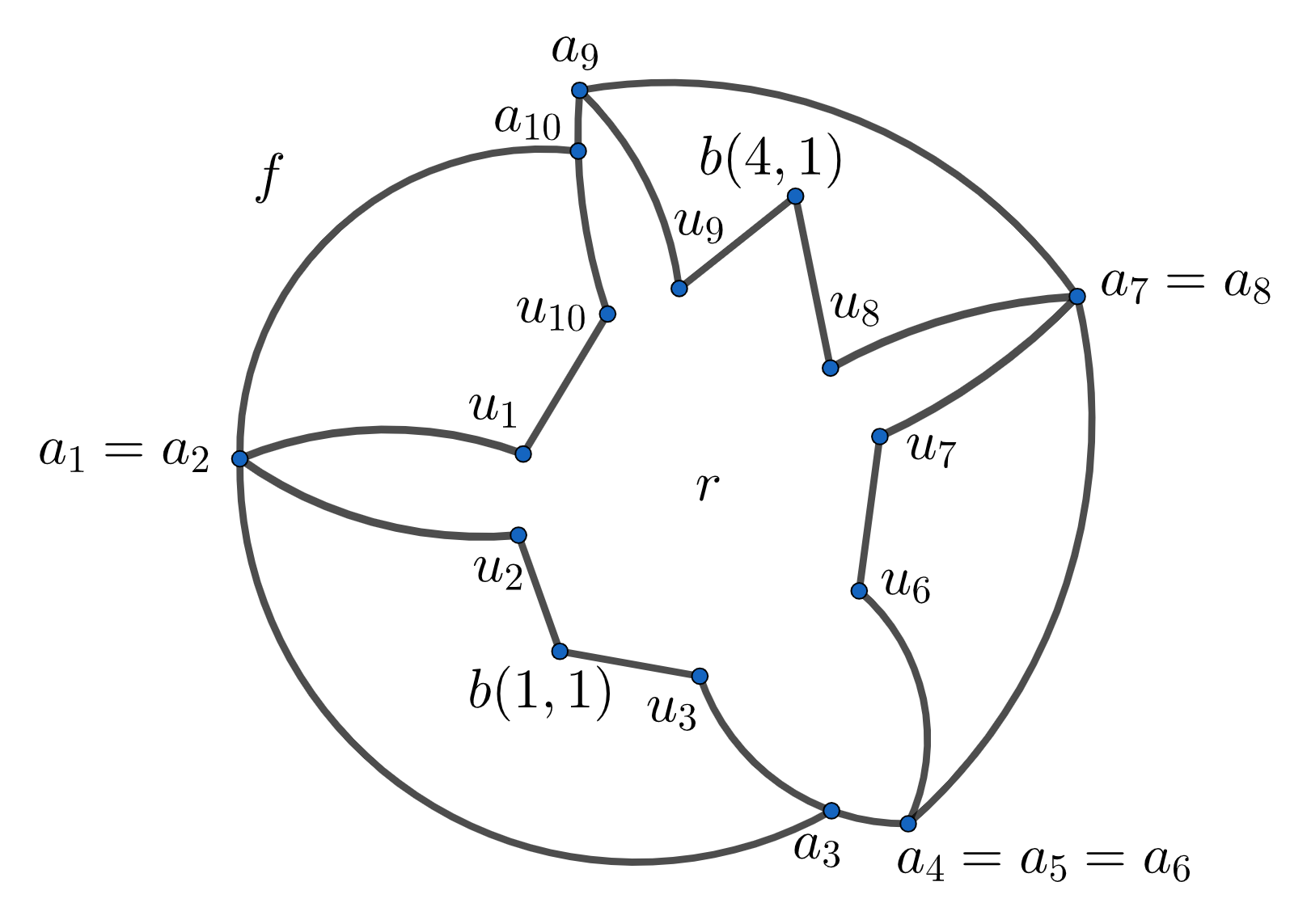}
	\caption{The graph $H'=H-u$, with $H$ as in Figure \ref{fig:aub1}. We have the $2$-cuts $\{a_2,a_3\}$, $\{a_6,a_7\}$, $\{a_8,a_9\}$, and $\{a_{10},a_1\}$, as stated in \eqref{eq:2cut}.}
	\label{fig:aub2}
\end{figure}

We write
\begin{equation}
	\label{eq:verta}
V(r\cap f)=\{a_1,a_2,\dots,a_{4D-2}\},
\end{equation}
where
\[V(f_i\cap f)=\{a_{2i-1},a_{2i}\}, \quad 1\leq i\leq 2D-1.\]
The vertices \eqref{eq:verta} are not necessarily all distinct, however, for all $1\leq i\leq 4D-2$ we have $a_{i}\neq a_{(i+4 \mod 4D-2)}$, otherwise three distinct faces of $H$ would each contain $u$ and $a_{i}$. We record that
\begin{empheq}[box=\fbox]{align}
\label{eq:2cut}
\notag&\text{for each } 1\leq i\leq 4D-2,
\text{ if } a_{2i}\neq a_{(2i+1 \bmod 4D-2)}, \\\notag&\text{then the set } \{a_{2i},a_{(2i+1 \bmod 4D-2)}\} \text{ is a } 2\text{-cut in } H',
\\&\text{and moreover } H'-a_{2i}-a_{(2i+1 \bmod 4D-2)} \text{ is bipartite}.
\end{empheq}

Among other vertices, the boundary of $r$ contains, in this order (up to relabelling)
\begin{align*}
&u_1,a_1,a_2,u_2,b(1,1),b(1,2),\dots,b(1,n_1),u_3,a_3,a_4,u_4,b(2,1),b(2,2),\dots,b(2,n_2),\\&\dots,u_{4D-3},a_{4D-3},a_{4D-2},u_{4D-2},b(2D-1,1),b(2D-1,2),\dots,b(2D-1,n_{2D-1}),
\end{align*}
where $n_1,n_2,\dots,n_{2D-1}\geq 0$,
\[u_{2\ell-1},u_{2\ell}\in V(f_\ell), \quad 1\leq \ell\leq 2D-1\]
(with $u_{2\ell-1},u_{2\ell}$ not necessarily distinct), and moreover
\[u_\ell, \quad 1\leq \ell\leq 4D-2 \qquad \text { and } \qquad b(j,k), \quad 1\leq j\leq 2D-1, \ 1\leq k\leq n_j\]
are the neighbours of $u$ in $H$ (Figures \ref{fig:aub1} and \ref{fig:aub2}).

By the way, if for some $1\leq i\leq 2D-1$ we have $a_{2i}=a_{(2i+1 \bmod 4D-2)}$, then in fact we have
\[a_{2i}=u_{2i}=b(i,1)=b(i,2)=\dots=b(i,n_i)=u_{(2i+1 \bmod 4D-2)}=a_{(2i+1 \bmod 4D-2)}.\]
If this happens for all indices $1\leq i\leq 2D-1$, then $H$ is simply the $2D-1$-gonal pyramid.

Back to the proof, by \eqref{eq:2cut}, $H'\wedge K_2$ is a planar, $2$-connected graph. It may be sketched in the plane such that
\begin{align*}
	&u_2y,b(1,1)y,b(1,2)y,\dots,b(1,n_1)y,u_3y,\\&u_4x,b(2,1)x,b(2,2)x,\dots,b(2,n_2)x,u_{5}x,\\&\dots,u_{4D-2}y,b(2D-1,1)y,b(2D-1,2)y,\dots,b(2D-1,n_{2D-1})y,u_1y,\\&u_2x,b(1,1)x,b(1,2)x,\dots,b(1,n_1)x,u_3x,\\&u_4y,b(2,1)y,b(2,2)y,\dots,b(2,n_2)y,u_{5}y,\\&\dots,u_{4D-2}x,b(2D-1,1)x,b(2D-1,2)x,\dots,b(2D-1,n_{2D-1})x,u_1x
\end{align*}
appear, in this order, around the boundary of a region (Figure \ref{fig:aub3}).
\begin{figure}[h!]
	\centering
	\includegraphics[width=8.5cm]{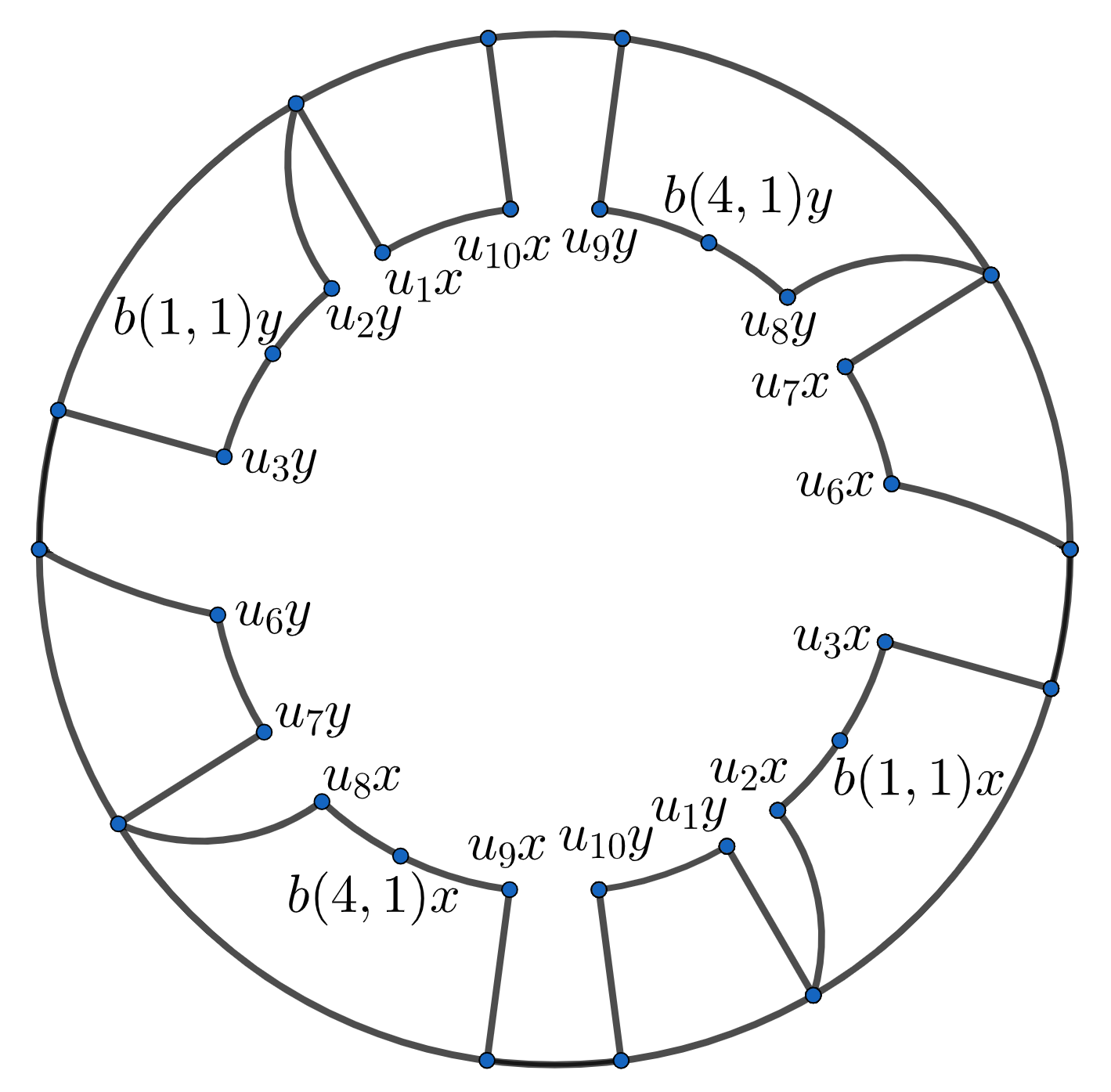}
	\caption{Sketch of $H'\wedge K_2$, with $H'$ as in Figure \ref{fig:aub2}.}
	\label{fig:aub3}
\end{figure}

In fact, due to \eqref{eq:2cut}, we can instead draw $H'\wedge K_2$ so that
\begin{align}
	\notag&u_2x,b(1,1)x,b(1,2)x,\dots,b(1,n_1)x,u_3x,\\\notag&u_6x,b(3,1)x,b(3,2)x,\dots,b(3,n_3)x,u_{7}x,\\\notag&\dots,u_{4D-2}x,b(2D-1,1)x,b(2D-1,2)x,\dots,b(2D-1,n_{2D-1})x,u_1x,\\\notag&u_4x,b(2,1)x,b(2,2)x,\dots,b(2,n_2)x,u_5x,\\\notag&u_8x,b(4,1)x,b(4,2)x,\dots,b(4,n_4)x,u_{9}x,\\&\dots,u_{4D-4}x,b(2D-2,1)x,b(2D-2,2)x,\dots,b(2D-2,n_{2D-2})x,u_{4D-3}x
	\label{eq:uy}
\end{align}
appear, in this order, around the boundary of one region, while
\begin{align}
	\notag&u_2y,b(1,1)y,b(1,2)y,\dots,b(1,n_1)y,u_3y,\\\notag&u_6y,b(3,1)y,b(3,2)y,\dots,b(3,n_3)y,u_{7}y,\\\notag&\dots,u_{4D-2}y,b(2D-1,1)y,b(2D-1,2)y,\dots,b(2D-1,n_{2D-1})y,u_1y,\\\notag&u_4y,b(2,1)x,b(2,2)y,\dots,b(2,n_2)y,u_5y,\\\notag&u_8y,b(4,1)y,b(4,2)y,\dots,b(4,n_4)y,u_{9}y,\\&\dots,u_{4D-4}y,b(2D-2,1)y,b(2D-2,2)y,\dots,b(2D-2,n_{2D-2})y,u_{4D-3}y
	\label{eq:ux}
\end{align}
appear, in this order, around the boundary of another region (Figure \ref{fig:aub4}).
\begin{figure}[h!]
	\centering
	\includegraphics[width=9cm]{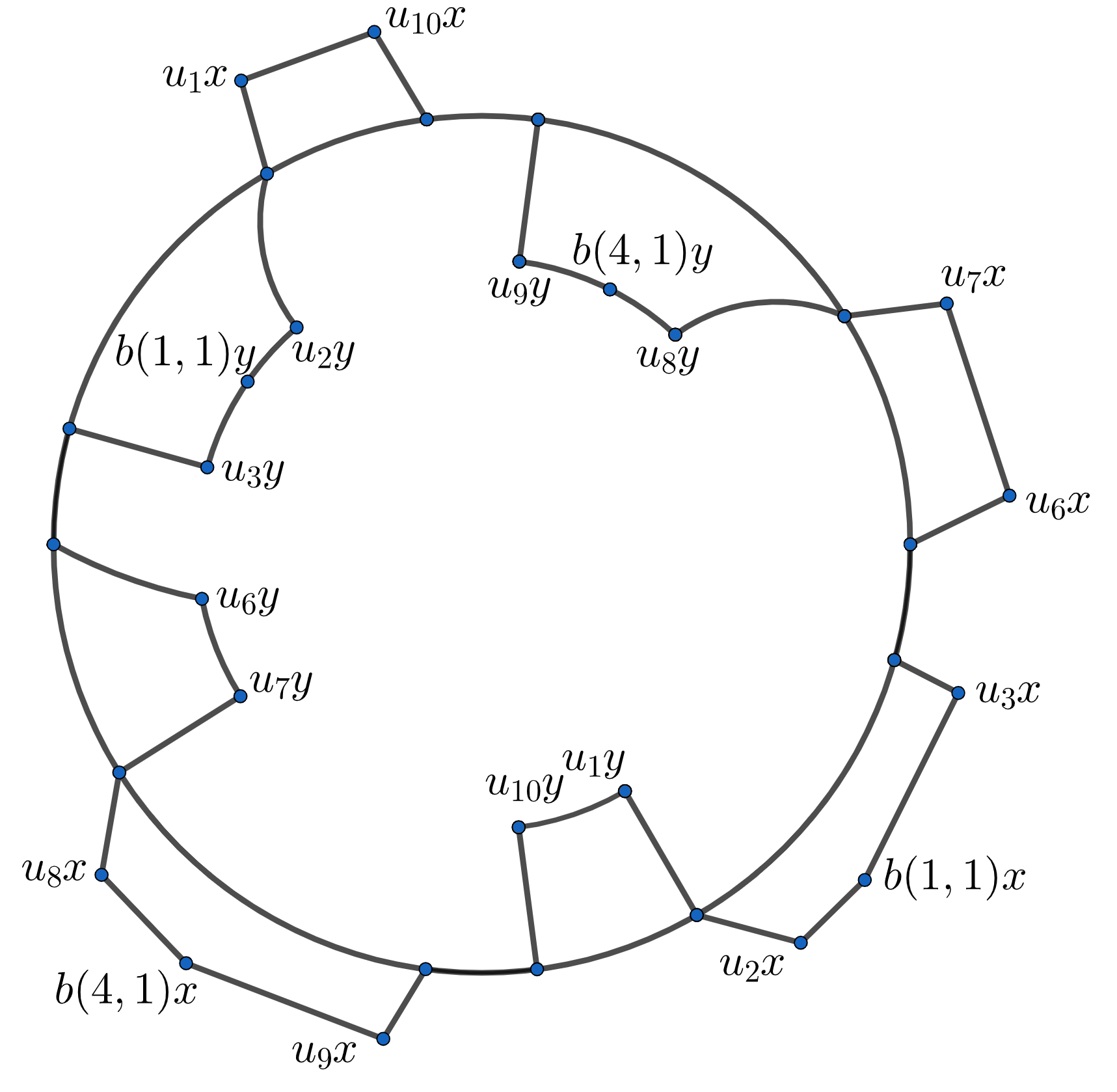}
	\caption{A planar immersion of $H'\wedge K_2$ (Figure \ref{fig:aub3}), that can be extended to a planar immersion of $H\wedge K_2$.}
	\label{fig:aub4}
\end{figure}

Finally, we add to $H'\wedge K_2$ the vertices $ux$ and $uy$, together with edges between $uy$ and each vertex of \eqref{eq:uy}, and edges between $ux$ and each vertex of \eqref{eq:ux}. By construction, $H\wedge K_2$ is indeed a $3$-polytope.
\end{proof}

\section{Proof of Theorems \ref{thm:3} and \ref{thm:4}}
\label{sec:thm3}
It suffices to prove Theorem \ref{thm:4}, since Theorem \ref{thm:3} is a special case of it.

\begin{proof}[Proof of Theorem \ref{thm:4}]
$\Rightarrow$. Let $\calP=H\wedge K_2$ be a $3$-polytope. We already know that $H$ is $2$-connected, of minimum vertex degree $3$. We delete edges from $H$ successively in the following way. At each step, after fixing an odd cycle $\calC$, there exists an edge $e\in\calC$ such that $H-e$ is still $2$-connected: we delete $e$. We continue in this fashion, stopping when we obtain either a bipartite graph, or simply an odd cycle. Later we will show that the scenario of ending up with an odd cycle is actually impossible. We reinsert each deleted edge that keeps the resulting graph bipartite. The obtained graph $H'$ is the sought bipartite, spanning subgraph of $H$. The set of edges that have been removed and not reinserted shall be denoted by
\[E'=\{a_1b_1,a_2b_2,\dots,a_mb_m\},\]
for some non-negative integer $m$. By construction, each $H'+a_ib_i$, $1\leq i\leq m$, is not bipartite.

Recall that
\begin{equation}
	\label{eq:cons2}
	H\wedge K_2=H'\dot\cup H'+(a_1,x)(b_1,y)+(a_1,y)(b_1,x)+\dots+(a_m,x)(b_m,y)+(a_m,y)(b_m,x).
\end{equation}
Since $\calP$ is $3$-connected, we have $m\geq 2$, and up to relabelling $a_1,b_1,a_m,b_m$ are distinct.

Since $\calP$ is planar, so is $H'\dot\cup H'$, hence $H'$ is planar. Again since $\calP$ is planar, all endpoints of elements in $E'$ lie on one region of $H'$, say $r$ (see Figure \ref{fig:fea}). We claim that distinct vertices
\[a_1,a_2,a_3,b_1,b_3,b_2\]
cannot lie on $r$ in this order (Figure \ref{fig:nfea}). Indeed, otherwise
\[\{\{a_1x,a_3x',b_2x''\},\{b_1y,b_3y',a_2y''\}\}\]
would determine an even subdivision of $K(3,3)$, where $x'=x$ and $y'=y$ if the distance on $r$ from $a_1$ to $a_3$ is even, and vice versa otherwise, and $x''=x$ and $y''=y$ if the distance on $r$ from $a_1$ to $a_2$ is odd, and vice versa otherwise. It follows that the endpoints of the elements in $E'$ lie on $r$ either in the order
\begin{equation}
	\label{eq:vertabm1}
	a_1,a_2,\dots,a_m,b_m,b_{m-1},\dots,b_1,
\end{equation}
or
\begin{equation}
	\label{eq:vertabm2}
	a_1,a_2,\dots,a_m,b_1,b_2,\dots,b_m.
\end{equation}
If $H$ is non-planar, then the order is \eqref{eq:vertabm2}.
\begin{figure}[h!]
	\centering
	\begin{subfigure}{0.65\textwidth}
		\centering
		\includegraphics[width=4.5cm]{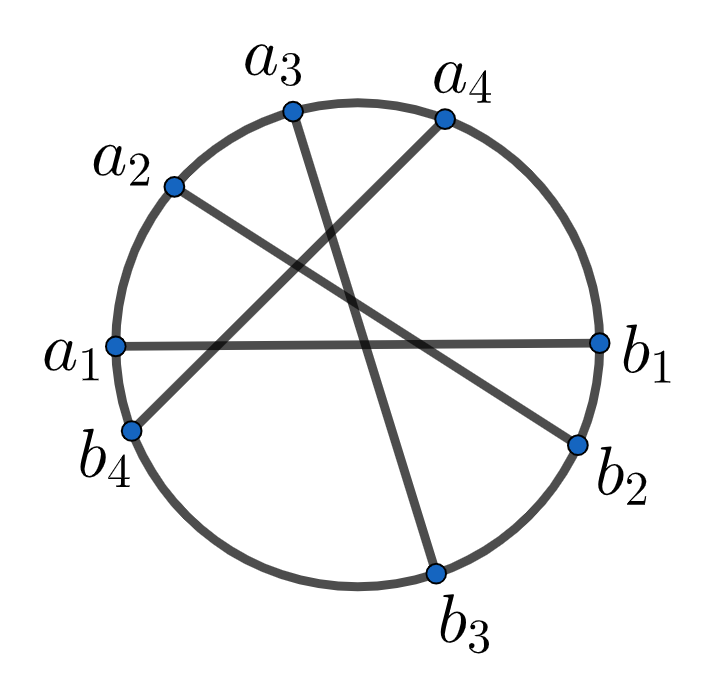}
		\includegraphics[width=4.0cm]{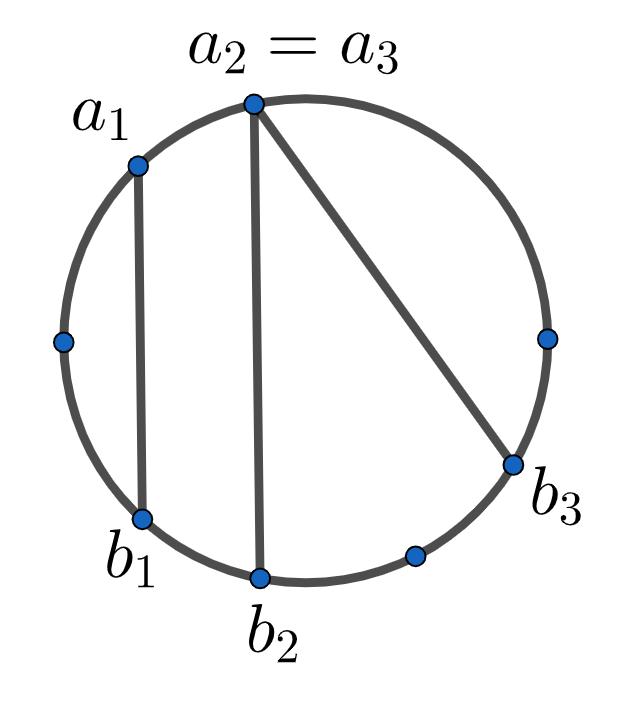}
		\caption{Two feasible orderings for endpoints of $E'$ around $r$.}
		\label{fig:fea}
	\end{subfigure}
	\begin{subfigure}{0.33\textwidth}
		\centering
		\includegraphics[width=4.75cm]{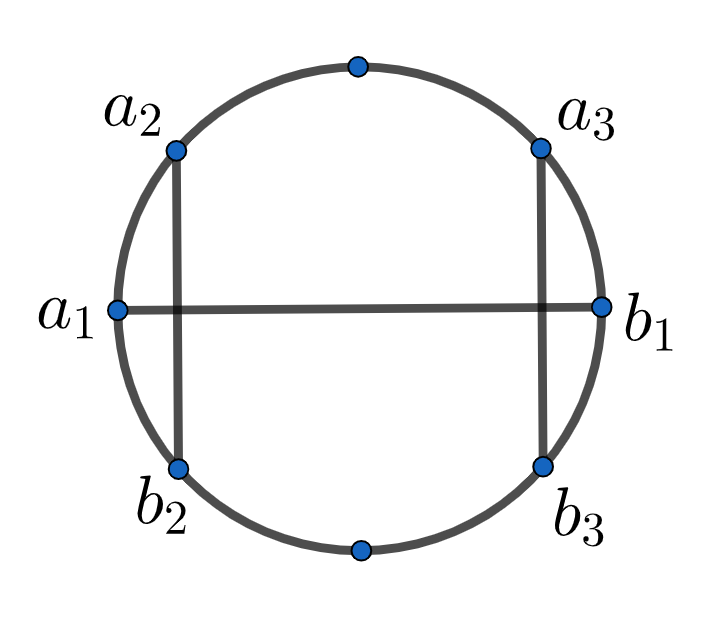}
		\caption{A non-feasible ordering.}
		\label{fig:nfea}
	\end{subfigure}
	\caption{Two feasible and one non-feasible ordering for endpoints of $E'$ around $r$.}
	\label{fig:fnf}
\end{figure}

Next, we apply Lemma \ref{lem:2cut}. Since $\calP$ is $3$-connected, $H'$ is either $3$-connected or semi-hyper-$2$-connected, and any $2$-cut belongs to the region $r$. Moreover, if \[v_1,v_2,\dots,v_{V},a_1,a_2,\dots,a_m,w_1,w_2,\dots,w_{W},b_1,b_2,\dots,b_m\]
lie in this order on the contour of $r$, then no $2$-cut set of $H'$ is a subset of
\[\text{either } \{b_m,v_1,v_2,\dots,v_{V},a_1\} \ \text{ or }\{a_m,w_1,w_2,\dots,w_{W},b_1\}.\]

It remains to show that, when $\calP=H\wedge K_2$ is a $3$-polytope, on removing edges from $H$ as described at the beginning of this proof, we can never end up with just an odd cycle. By contradiction, if we end up with just an odd cycle $\overline\calC$, then the edges of $H$ are either edges of $\overline\calC$, or elements of $E'$. Seeing as $\delta(H)\geq 3$, every vertex on $\overline\calC$ is an endpoint of an edge in $E'$, and in particular we are in scenario \eqref{eq:vertabm2}. Suppose that $a_1,a_2,a_3,b_1,b_2,b_3$ are distinct. Since $\overline\calC$ is odd, $H$ contains an even subdivision of a graph $H''$ that is obtained from $K(3,3)$ by subdividing one edge into a path on two edges (Figure \ref{fig:kpp}). We check that $H''\wedge K_2$ is non-planar, contradiction. Hence w.l.o.g. exactly two of $a_1,a_2,\dots,a_m$ are distinct. Since the length of $\overline\calC$ is at least $5$, then $m\geq 3$, and again w.l.o.g. say $a_1=a_2$, so that $b_1,b_2,b_m$ are distinct. Now
\[\{a_1a_m,a_1b_1,a_1b_2,a_1b_m,a_mb_1,b_1b_2,b_ma_m\}\subset E(H),\]
and moreover there is a $b_2b_m$-path on an even number of vertices along $\overline\calC$, not containing any of $a_1,a_m,b_1$ (Figure \ref{fig:4pyr}). Therefore, $H$ contains an even subdivision of the square pyramid, contradicting Corollary \ref{cor:4pyr}.
\begin{figure}[h!]
	\centering
	\begin{subfigure}{0.33\textwidth}
		\centering
		\includegraphics[width=3cm]{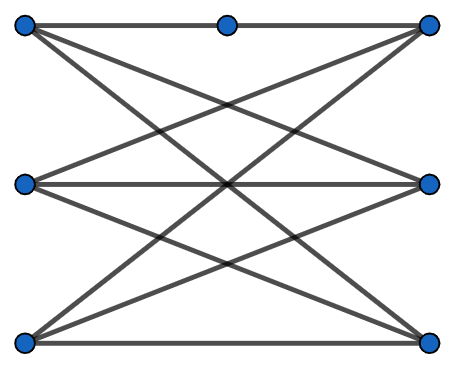}
		\caption{The graph $H''$.}
		\label{fig:kpp}
	\end{subfigure}
	\begin{subfigure}{0.65\textwidth}
		\centering
		\includegraphics[width=4.5cm]{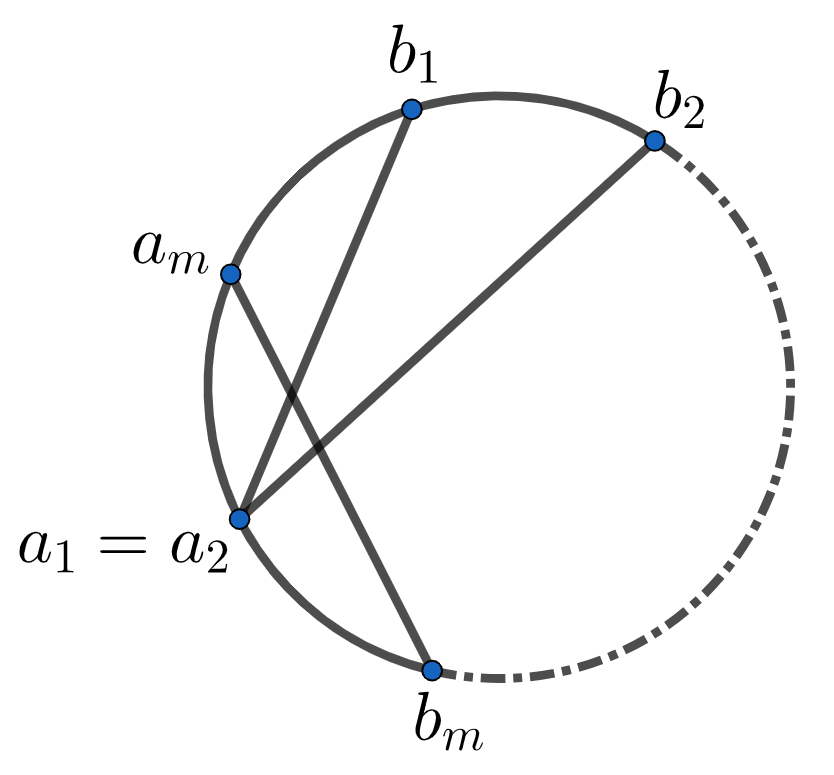}
		\caption{Exactly two of $a_1,a_2,\dots,a_m$ are distinct. Solid lines are edges, and the dash-dotted line a path of even length.}
		\label{fig:4pyr}
	\end{subfigure}
	\caption{Ruling out the scenario where the procedure of removing edges from $H$ terminates simply with an odd cycle $\overline\calC$.}
	\label{fig:oddc}
\end{figure}


$\Leftarrow$. Vice versa, assume the conditions on $H$ of Theorem \ref{thm:4}. Recalling the construction \eqref{eq:cons2}, we see that $H\wedge K_2$ is planar. By contradiction, let $\{ax,bx'\}$ be a $2$-cut in $H\wedge K_2$, where $x'=x$ or $y$. By the construction \eqref{eq:cons2}, and since $H'$ is $2$-connected, $\{a,b\}$ is a $2$-cut in $H$. By assumption, $\{a,b\}$ lies on the region $r$. By assumption and construction, there cannot in fact exist any such $2$-cut in $H\wedge K_2$.
\end{proof}

\appendix
\section{Proof of Propositions \ref{prop:cart} and \ref{prop:strong}}
\label{sec:appa}
Here we prove the characterisation of $3$-polytopal Cartesian and strong products, as stated in section \ref{sec:cs}.

\begin{proof}[Proof of Proposition \ref{prop:cart} -- cf. \cite{behmah}.]
	Let $\calP=A\square B$ be a $3$-polytopal graph. Then $A,B$ are connected. The two factors cannot both contain a cycle, because $K_3\square K_3$ is non-planar. Thus w.l.o.g. $A$ is a tree.
	
	We claim that $B$ has no vertex of degree $1$. By contradiction, call respectively $a,b$ vertices of $A,B$ of degree $1$. Then $\calP$ would have a vertex of degree $2$, namely $(a,b)$, contradiction. In particular, $B$ is not a tree, hence $B$ contains a cycle.
	
	Any tree is a path if and only if it does not contain a copy of $K(1,3)$. Since $B$ is cyclic and $K(1,3)\square K_3$ is not planar, we deduce that $A$ is in fact a path.
	
	One possibility is for $B$ to be a polygon. Then $\calP$ is a $3$-polytope for any choice $\geq 2$ of length for the path $A$ (i.e., $\calP$ is a stacked prism -- for example see Figure \ref{fig:cart2}).
	
	Henceforth, assume that $B$ is not a polygon. Call $T_1,T_2$ respectively the graph with two triangles sharing an edge (diamond graph) and the one with two triangles sharing a vertex. Since $B$ has no degree $1$ vertices, then either $B$ has a subgraph homeomorphic to $T_1$ or $T_2$, or $B$ contains two disjoint cycles. Note that \[T_1\square K(1,2) \qquad \text{and} \qquad T_2\square K(1,2)\]
	are non-planar. Moreover, if $B$ contains two disjoint cycles, then these are connected by a path, because $B$ is connected. Now $T_3\square K(1,2)$ is also non-planar, where $T_3$ is the graph obtained by adding any edge to $K_3\dot\cup K_3$. Therefore, in any case $A=K_2$.
	
	We record that $K_2\square K_4$ and $K_2\square K(2,3)$ are non-planar. But then $B$ contains no subgraph homeomorphic to $K_4$ or $K(2,3)$: in other words, $B$ is outerplanar. Moreover, if $B$ had a separating vertex $s$, then $\calP$ would not be $3$-connected, as
	\[P-(a,s)-(a',s)\]
	would be disconnected, where $a,a'$ are the vertices of $A$. Thereby, $B$ is $2$-connected. In the other direction, a moment of thought reveals that if $B$ is outerplanar and $2$-connected, i.e., a polygon with some added diagonals, then $K_2\square B$ is a $3$-polytope.
\end{proof}

\begin{proof}[Proof of Proposition \ref{prop:strong}]
Jha and Slutzki \cite{jhaslu} showed that, for connected $H,J$, the graph $H\boxtimes J$ is planar if and only if either $H=J=P_3$, or $J=K_2$ and $H$ is a tree. In the first case, $P_3\boxtimes P_3$ is a $3$-polytope. In the second case, we note that, if $a$ is a separating vertex of $H$, then $\{(a,x), (a,y)\}$ is a $2$-cut in $H\boxtimes K_2$. Therefore, the only other strong product that is a candidate for being a $3$-polytope is $K_2\boxtimes K_2$, i.e. the tetrahedron.
\end{proof}

\bibliographystyle{abbrv}
\bibliography{biblio}
\end{document}